\documentclass[12pt,reqno]{amsart}

\usepackage{amssymb,verbatim}
\usepackage{amsmath,amsfonts}
\usepackage[mathscr]{euscript} 
\usepackage{amsthm}
\usepackage{color}
\usepackage{url}
\usepackage{graphicx} 
\usepackage{enumitem} 
\usepackage{lineno,xcolor} 
\usepackage{hyperref} 
\hypersetup{colorlinks} 
\usepackage{bbm} 
\usepackage{mathrsfs} 

\usepackage[lined,algonl,boxed,norelsize]{algorithm2e} 
\usepackage{diagbox}
\usepackage{threeparttable} 
\usepackage{extarrows} 
\usepackage{tikz-cd} 
\usepackage{ltablex} 

\usepackage{graphicx} 

\usepackage{standalone} 

\setlist[enumerate]{
label=\textnormal{({\roman*})},
ref={\roman*}}

\makeatletter
\def\th@plain{%
  \thm@notefont{}
  \itshape 
}
\def\th@definition{%
  \thm@notefont{}
  \normalfont 
}
\makeatother

\makeatletter
\newtheorem*{rep@theorem}{\rep@title}
\newcommand{\newreptheorem}[2]{%
\newenvironment{rep#1}[1]{%
 \def\rep@title{#2 \ref{##1}}%
 \begin{rep@theorem}}%
 {\end{rep@theorem}}}
\makeatother

\newreptheorem{theorem}{Theorem}
\newreptheorem{corollary}{Corollary}

\newtheorem{theorem}{Theorem}[section]
\newtheorem{lemma}[theorem]{Lemma}
\newtheorem{proposition}[theorem]{Proposition}

\newtheorem{corollary}[theorem]{Corollary}

\theoremstyle{remark}
\newtheorem*{remark}{Remark}

\theoremstyle{definition}

\newtheorem{examplex}[theorem]{Example}
\newenvironment{example}
  {\pushQED{\qed}\examplex}
  {\popQED\endexamplex}

  \newtheorem{definitionx}[theorem]{Definition}
\newenvironment{definition}
  {\pushQED{\qed}\definitionx}
  {\popQED\enddefinitionx}

\makeatletter
\newcommand*{\da@rightarrow}{\mathchar"0\hexnumber@\symAMSa 4B }
\newcommand*{\da@leftarrow}{\mathchar"0\hexnumber@\symAMSa 4C }
\newcommand*{\xdashrightarrow}[2][]{%
  \mathrel{%
    \mathpalette{\da@xarrow{#1}{#2}{}\da@rightarrow{\,}{}}{}%
  }%
}
\newcommand{\xdashleftarrow}[2][]{%
  \mathrel{%
    \mathpalette{\da@xarrow{#1}{#2}\da@leftarrow{}{}{\,}}{}%
  }%
}
\newcommand*{\da@xarrow}[7]{%
  \sbox0{$\ifx#7\scriptstyle\scriptscriptstyle\else\scriptstyle\fi#5#1#6\m@th$}%
  \sbox2{$\ifx#7\scriptstyle\scriptscriptstyle\else\scriptstyle\fi#5#2#6\m@th$}%
  \sbox4{$#7\dabar@\m@th$}%
  \dimen@=\wd0 %
  \ifdim\wd2 >\dimen@
    \dimen@=\wd2 %
  \fi
  \count@=2 %
  \def\da@bars{\dabar@\dabar@}%
  \@whiledim\count@\wd4<\dimen@\do{%
    \advance\count@\@ne
    \expandafter\def\expandafter\da@bars\expandafter{%
      \da@bars
      \dabar@ 
    }%
  }%
  \mathrel{#3}%
  \mathrel{%
    \mathop{\da@bars}\limits
    \ifx\\#1\\%
    \else
      _{\copy0}%
    \fi
    \ifx\\#2\\%
    \else
      ^{\copy2}%
    \fi
  }%
  \mathrel{#4}%
}
\makeatother

    \makeatletter 
\newcommand{\overrightharpoon}{%
  \mathpalette{\overarrow@\rightharpoonfill@}}
\def\rightharpoonfill@{\arrowfill@\relbar\relbar\rightharpoonup}
\makeatother

\makeatletter 
\newcommand{\osh}{\mathpalette{\overarrowsmall@\rightharpoonfill@}}
\def\rightharpoonfill@{\arrowfill@\relbar\relbar\rightharpoonup}
\newcommand{\overarrowsmall@}[3]{%
  \vbox{%
    \ialign{%
      ##\crcr
      #1{\smaller@style{#2}}\crcr
      \noalign{\nointerlineskip}%
      $\m@th\hfil#2#3\hfil$\crcr
    }%
  }%
}
\def\smaller@style#1{%
  \ifx#1\displaystyle\scriptstyle\else
    \ifx#1\textstyle\scriptstyle\else
      \scriptscriptstyle
    \fi
  \fi
}
\makeatother

\makeatletter
\newcommand{\mylabel}[2]{#2\def\@currentlabel{#2}\label{#1}}
\makeatother

\DeclareMathOperator{\Eb}{\mathbb{E}} 
\DeclareMathOperator{\FV}{\textnormal{FV}} 
\DeclareMathOperator{\Gc}{\mathcal{G}} 
\DeclareMathOperator{\LV}{\textnormal{LV}} 
\DeclareMathOperator{\ousf}{\osh{\mathsf{USF}}} 
\DeclareMathOperator{\owusf}{\osh{\mathsf{WUSF}}} 
\DeclareMathOperator{\Pb}{\mathbb{P}} 
\DeclareMathOperator{\Rb}{\mathbb{R}} 
\DeclareMathOperator{\SF}{\osh{\textnormal{SF}}} 
\DeclareMathOperator{\Tb}{\mathbb{T}} 
\DeclareMathOperator{\Zb}{\mathbb{Z}} 

\title[Rotor walks and the wired spanning forest]{Rotor walks on transient graphs and the wired spanning forest}
\author{Swee Hong Chan}
 \thanks{Department of Mathematics, Cornell University. Partially supported by NSF grant DMS-1455272. Email: \url{sweehong@math.cornell.edu}.}

\begin{document}

\begin{abstract}
We study rotor walks on transient graphs  with initial rotor configuration  sampled from the oriented wired uniform spanning forest (OWUSF) measure.
We show that the expected number of visits to any vertex by the rotor walk is at most equal to the expected number of visits by the simple random walk.
In particular, this implies that this  walk is transient.
When these two numbers coincide,
we show that the  rotor configuration at the end of the process  also has the law of OWUSF.
Furthermore,  if the  graph is vertex-transitive,  we show that the average number of visits  by $n$ consecutive rotor walks converges to the Green function of the simple random walk as $n$ tends to infinity.
This answers a question posed by Florescu, Ganguly, Levine, and Peres (2014).
\end{abstract}



\maketitle

\section{Introduction}\label{section: intro}
In a  \emph{rotor walk}~\cite{WLB96,PDDK96, Propp03}  on a graph $G$,
each vertex  is assigned a fixed cyclic ordering of its neighbors, and  each vertex has a \emph{rotor}, which is an arrow that points to one of its neighbors.
A \emph{rotor configuration} is an assignment of  directions to all the rotors.
Given an initial rotor configuration,
a walker (initially located at a fixed vertex) explores the graph using the following rule:
at each time step, the walker changes the rotor of its current location to point to the next neighbor given by the  cyclic ordering, and then the walker moves to this new neighbor.
The rotor walk is obtained by repeated applications of this rule.

One major difference of this paper compared to other works in the literature is our choice of  initial rotor configuration;
it is sampled from  the {oriented wired uniform spanning forest} measure.
   Let $G$ be a connected graph that is simple (i.e. no loops or multiple edges), transient, 
  and locally finite (i.e. every vertex has finite degree), and let
 $W_1\subseteq W_2 \subseteq \ldots$ be finite connected subsets of $V(G)$  such that $\bigcup_{R=1}^\infty W_R=V(G)$.
Let $G_R$ be obtained from $G$ by identifying all vertices outside $W_R$ to one new vertex $w_R$, and let $\mu_R$ be the uniform measure on spanning trees of $G_R$ oriented toward $w_R$.
Then $\mu_R$ has a unique infinite volume limit \cite{Pem91,BLPS01}, which 
we call the \emph{wired spanning forest oriented toward infinity} $\owusf(G)$.
See \cite{BLPS01,LP16} for more details.

Several studies had been conducted to   compare  the behavior of rotor walks to the expected behavior of simple random walks (e.g.\cite{CDST06,CS06,LL09,LP09,HMSH15,HSH18}).
One such  result is due to Schramm \cite[Theorem~10]{HP10}, who showed that the rotor walk is in a certain sense at most as transient as the simple random walk.
We will show  that that the opposite is true when the initial rotor configuration is given by $\owusf(G)$, in a manner to be made precise.

One way to measure the transience of rotor walks is to count the number of visits to any vertex.
Fix a vertex $a$ as the  initial location of the walker.
Let $u(\rho,x)$ be the number of visits to $x \in V(G)$ by the rotor walk with initial rotor configuration $\rho$,
and let $\Gc(x)$ be the expected number of visits to $x$ by the simple random walk.
Note that $\Gc(x)$ is finite since the graph $G$ is transient.

\begin{theorem}\label{theorem: odometer is bounded above by Green function}
Let $G$ be a simple connected graph that is locally finite and transient.
Consider any rotor walk on $G$ with the walker initially located at a fixed vertex $a$.
Then, 
\begin{equation}\label{equation: odometer is bounded above by Green function}
\Eb_{\rho}[u(\rho,x)]\leq \Gc(x) \qquad \forall \ x \in V(G), 
\end{equation}
where $\rho$ is sampled from $\owusf(G)$.
\end{theorem}
We prove  Theorem~\ref{theorem: odometer is bounded above by Green function} by first proving an analogous statement for finite graphs,
and the statement for infinite graphs then follows by taking the infinite volume limit.

One consequence of Theorem~\ref{theorem: odometer is bounded above by Green function} is that  the rotor walk with initial rotor configuration picked from $\owusf(G)$ is transient (i.e., every vertex is visited only finitely many times) almost surely.
We remark that the rotor walk with an arbitrary initial rotor configuration can fail to be transient even if the underlying graph is transient; see \cite[Theorem~2]{AH12}.

Note that  the inequality in \eqref{equation: odometer is bounded above by Green function} can be strict, as shown in Figure~\ref{figure: odometer bound strict inequality} with $G$ 
being  a transient tree with an extra infinite path attached to the root.
Somewhat surprisingly,  having  equality in \eqref{equation: odometer is bounded above by Green function} turns out to have the following interesting implication.

\begin{figure}[ht!]
\centering
\begin{tabular}{c  }
\includegraphics[scale=1]{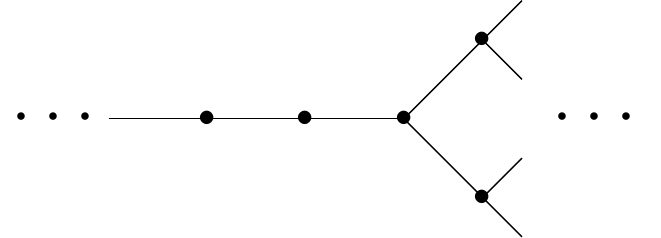}  \\
(a) 
\\ \vspace{0.1 cm}\\
\includegraphics[scale=1]{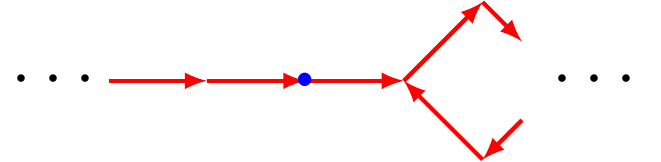}  \\
(b)\\
\\ \vspace{0.1 cm}\\
\includegraphics[scale=1]{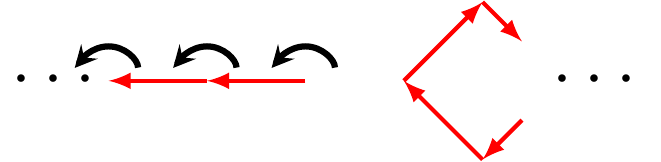}  \\
(c)
\end{tabular}
\caption{(a) The $2$-ary tree with an extra infinite path attached to its root.
 (b) An initial rotor configuration $\rho$ sampled from $\owusf(G)$, and a walker at the initial location $a$, marked with a (blue) bullet.
 The rotors of $\rho$ in the extra infinite path form a path oriented toward $a$ almost surely 
 by Wilson's method~\cite{BLPS01}.
 (c) The final rotor configuration $\sigma(\rho)$ after the rotor walk is performed.
 The number of visits $u(\rho,a)$ to $a$ is equal to 1 almost surely as $a$ is visited only once (i.e., at the beginning at the walk),
 while  the Green function $\Gc(a)$ is equal to $3$ (see \cite[Exercise~2.8]{LP16}).
 Furthermore, $\rho$ and $\sigma(\rho)$ follow different laws
 as the former has an infinite path oriented toward $a$ almost surely 
 while the latter has the same  path oriented outward of $a$ almost surely.
 }
 \label{figure: odometer bound strict inequality}
\end{figure}

Let $\rho$ be a rotor configuration such that the corresponding rotor walk is transient.
Then  the \emph{final rotor configuration} $\sigma(\rho)$  is given by  $\sigma(\rho)(x):=\lim_{t \to \infty} \rho_t(x)$.
Here $\rho_t$ denotes the rotor configuration at the $t$-th step of the rotor walk.
Note that the limit $\lim_{t \to \infty} \rho_t(x)$ exists as the sequence is eventually constant.
A probability measure $\mu$ on rotor configurations is \emph{stationary with respect to the rotor walk}
if $\rho \overset{d}{=} \mu$ implies $\sigma(\rho)\overset{d}{=}\mu$.

\begin{theorem}\label{theorem: two TFAE}
Let $G$ be a simple connected graph that is locally finite and transient.
Consider any rotor walk on $G$ with the walker initially located at a fixed vertex $a$.
Let $\rho$ be sampled from $\owusf(G)$.
Then, the following are equivalent:
\begin{enumerate}
\item $\owusf(G)$ is stationary with respect to the rotor walk; 
\item We have $\Eb_{\rho}[u(\rho,x)]= \Gc(x)$ for all $x \in V(G)$.
\end{enumerate}
\end{theorem}
The proof of Theorem~\ref{theorem: two TFAE} uses an idea similar to Theorem~\ref{theorem: odometer is bounded above by Green function};
we first show an analogous statement for the finite graphs, and then we take the infinite volume limit.
This  limit needs to be taken over a sequence of random variables that are tight (as otherwise the equality in (ii)  will we weakened to an inequality), and 
this tightness condition turns out to be equivalent to requiring the stationarity of $\owusf(G)$.
A more detailed sketch is provided in Section~\ref{section: rotor walk stationarity}.

See Proposition~\ref{proposition: tree stationarity} 
for graphs  for which $\owusf(G)$ is  stationary with respect to the rotor walk.  
Those examples include the  $b$-ary tree $\Tb_b$ for $b\geq 2$
(i.e. a tree with a root vertex $o$ having  degree $b$ and with every other vertex having degree $b+1$).
For the other end of the spectrum,
see Figure~\ref{figure: odometer bound strict inequality} for  a graph for which $\owusf(G)$ is not stationary.
We remark that the  stationarity of $\owusf(\Zb^d)$ for rotor walks on $\Zb^d$ ($d \geq 3$) remains an open problem; see Section~\ref{section: open problems}.

Another way to measure the transience of  rotor walks is the following method introduced by Florescu, Ganguly, Levine, and Peres (FGLP)~\cite{FGLP14}:
Start with an initial rotor configuration and with $n$ walkers located at the fixed  vertex $a$.
Let each of these $n$ walkers in turn perform rotor walk
(note that we do not reset the rotors in between runs!).
Let $S_n(\rho,x)$ be equal to  the total number of visits to $x$ by  all the walkers if all of the $n$  rotor walks  are transient, and is equal to infinity otherwise.
 The \emph{occupation rate} $S_n(\rho,x)/n$ satisfies the following inequality, 
\begin{equation}\label{equation: Schramm bound}
 \liminf_{n \to \infty} \frac{S_n(\rho,x)}{n}\geq  \Gc(x).  
\end{equation}
The proof of \eqref{equation: Schramm bound} for  when  $x$ is equal to the initial vertex $a$
 is due to Schramm (see  \cite[Theorem~10]{HP10} and \cite[Section~2]{FGLP14}).
  Note that Schramm  stated \eqref{equation: Schramm bound} in terms of the escape rate of the rotor walk, which is inversely proportional to ${S_n(\rho,a)}/{n}$; see \cite[Lemma~5]{FGLP14}.
   We include a proof of \eqref{equation: Schramm bound} in  this paper for completeness; see Lemma~\ref{lemma: Schramm bound}. 

%


The inequality in \eqref{equation: Schramm bound} can be strict; see \cite[Theorem~2(iii)]{AH11}.
FGLP then asked for the next best thing:  must there always exist a rotor configuration for which equality occurs in \eqref{equation: Schramm bound}?
We give a positive answer to a weaker probabilistic variant of this question.

\begin{theorem}\label{theorem: Schramm bound convergence in norm}
Let $G$ be a simple connected graph that is locally finite, transient, and vertex-transitive.
Consider any rotor walk on $G$ with the walker initially located at a fixed vertex $a$.
Let $\rho$ be sampled from $\owusf(G)$.
Then  occupation rates $S_n(\rho,x)/n$ converge in norm to $\Gc(x)$,
i.e.,
\[  \lim_{n \to \infty} \Eb_{\rho} \left[ \left|\frac{S_n(\rho,x)}{n}-\Gc(x)\right| \right] = 0 \qquad \forall x \in V(G).\]
\end{theorem}
The proof of  Theorem~\ref{theorem: Schramm bound convergence in norm} is derived from  
an upper bound for the expected value of occupation rates that holds  if $\rho \overset{d}{=} \owusf(G)$ and 
the lower bound for occupation rates from \eqref{equation: Schramm bound} that holds for any $\rho$.

When the underlying graph is vertex-transitive, we can upgrade the convergence in norm in Theorem~\ref{theorem: Schramm bound convergence in norm} to the almost sure convergence
and 
 gives a positive answer to the question of FGLP.


\begin{theorem}\label{theorem: Schramm bound for vertex-transitive graphs}
Let $G$ be a simple connected graph that is locally finite, transient, and vertex-transitive.
Consider any rotor walk on $G$ with the walker initially located at a fixed vertex $a$.
Then, for almost every $\rho$ picked from $\owusf(G)$,
\[  \lim_{n \to \infty} \frac{S_n(\rho,x)}{n}=\Gc(x) \qquad \forall \ x \in V(G).\]
\end{theorem}

The proof of Theorem~\ref{theorem: Schramm bound for vertex-transitive graphs} is inspired by  Etemadi's proof of strong law of large numbers~\cite{Ete81}.
We  first  estimate the probability $q_n$
that $S_n/n$ differs from $\Gc$ by more than $\epsilon$.
We then show that  the sum $\sum q_n$ is finite when summed over any subsequence  $n_1,n_2,\ldots$ that grows exponentially, and 
 by Borel-Cantelli lemma we then conclude that $S_n/n$ converges for these subsequences.
We then upgrade this convergence to the whole sequence by using the inequality
\[  \left( \frac{n_k}{n_{k+1}}\right) \frac{S_{n_k}}{n_k}  \ \leq \ \frac{S_n}{n} \ \leq \ \left( \frac{n_{k+1}}{n_{k}}\right) \frac{S_{n_{k+1}}}{n_{k+1}},\]
which holds for any $n \in [n_k,n_{k+1}]$ ($k\geq 1$).

The crucial step here is the  estimate of  $q_n$,
which  uses an upper bound for occupation rates that hold if $\rho \overset{d}{=} \owusf(G)$ and a quantitative version of \eqref{equation: Schramm bound} that gives a lower bound for occupation rates in terms of the volume growth of $G$.
The volume growth of $G$ can in turn be estimated by using the work \cite{SC95,Tro03} that holds for all vertex-transitive graphs.


We now present another scenario for which we can give a positive answer to the question of FGLP.
\begin{theorem}\label{theorem: Schramm's bound for rotor walk stationary graphs}
Let $G$ be a connected simple graph that is locally finite and transient. 
Consider any rotor walk on $G$ with the walker initially located at a fixed vertex $a$.
Suppose that  $\owusf(G)$ is stationary with respect to the given rotor walk.
Then, for almost every $\rho$ picked from $\owusf(G)$,
\[  \lim_{n \to \infty} \frac{S_n(\rho,x)}{n}=\Gc(x) \qquad \forall \ x \in V(G).\]
\end{theorem}
The proof of Theorem~\ref{theorem: Schramm's bound for rotor walk stationary graphs} uses the pointwise ergodic theorem to derive the almost sure convergence.
Note that  we can use  the pointwise ergodic theorem because the initial rotor configuration is stationary with respect to the rotor walk. 

The  question of FGLP has previously been answered positively  for all trees by Angel and Holroyd~\cite{AH11}
and for $\Zb^d$ by He~\cite{He14}.
In both cases, Theorem~\ref{theorem: Schramm bound for vertex-transitive graphs} (for $\Zb^d$) and Theorem~\ref{theorem: Schramm's bound for rotor walk stationary graphs} (for $\Tb_b$) provide  new examples of rotor configurations that answer the question of FGLP.
For any other vertex-transitive graph, Theorem~\ref{theorem: Schramm bound for vertex-transitive graphs} is the first one to provide an answer to this  question  to the best of our knowledge.

This paper is structured as follows.
In Section~\ref{section: preliminaries} we review notations  for rotor walks that will be used throughout this paper.
In Section~\ref{section: RWLMs for finite graphs} we review basic results for rotor walks on finite graphs.
In Section~\ref{section: transience of rotor walks} we prove 
Theorem~\ref{theorem: odometer is bounded above by Green function}.
In Section~\ref{section: convergence in norm of occupation rate} we prove Theorem~\ref{theorem: Schramm bound convergence in norm}.
In Section~\ref{section: rotor walk stationarity} we prove Theorem~\ref{theorem: two TFAE}.
In Section~\ref{section: rotor walk for b-ary trees} we provide some examples of graphs for which $\owusf(G)$ is stationary with respect to the rotor walk.
In Section~\ref{section: occupation rate} we prove 
Theorem~\ref{theorem: Schramm bound for vertex-transitive graphs} and Theorem~\ref{theorem: Schramm's bound for rotor walk stationary graphs}.
In Section~\ref{section: open problems} we list some open problems.

\begin{remark} Most of our results hold for the more general setting of random walks with local memory (RWLM)~\cite{CGLL18},
where the update step for the rotor at any vertex $x$  is   determined by  a Markov chain $M_x$ assigned to $x$ (instead of the given cyclic ordering).
Here $M_x$ is an ergodic Markov chain  such that its state space is the neighbors $N(x)$ of $x$ and its stationary distribution is the uniform distribution on $N(x)$.
In particular, Theorem~\ref{theorem: odometer is bounded above by Green function}, \ref{theorem: two TFAE}, \ref{theorem: Schramm bound convergence in norm}, and \ref{theorem: Schramm's bound for rotor walk stationary graphs} hold for all RWLMs. 
Note that Theorem~\ref{theorem: Schramm bound for vertex-transitive graphs} does not immediately extend to all RWLMs as 
the estimate of $q_n$ used in the proof is exclusive to rotor walks.
\end{remark}

%

\section{Preliminaries}\label{section: preliminaries}

Throughout this paper  $G:=(V(G),E(G))$ is a connected simple undirected graph that is locally finite (i.e. every vertex has finitely many edges).

The \emph{rotor walk} $(X_t)_{t \geq 0}$   on $G$ is  defined as follows.
Fix a vertex $a \in V(G)$ and a subset $Z \subseteq V(G)$.
To each vertex $x\in V(G) \setminus Z$ we assign  a  \emph{local mechanism} $\tau_x$, which is a bijection on  the neighbors $N(x)$ of $x$.
We assume  that each $\tau_x$ has one unique orbit (i.e.  $\{\tau^i(y) \mid i\geq 0 \}= N(x)$ for any neighbor $y$ of $x$). 
A \emph{rotor configuration} of $G$ is a function $\rho:V(G) \setminus Z \to V(G)$ such that $\rho(x) \in N(x)$ for all $x \in V(G) \setminus Z$.


The walker is initially located at $a$ (i.e. $X_0:=a$) and with an initial rotor configuration $\rho_0:=\rho$.
At the $t$-th step of the walk, 
the rotor of the current location of the walker
is incremented to point to the next vertex in the cyclic order 
 specified by its local mechanism, and then the walker moves to the vertex specified by this new rotor.
That is to say,
\begin{equation}\label{equation: rotor walk update}
\begin{split}
\rho_{t+1}(x)&:=\begin{cases}
 \rho_t(x) & \text{ if } x \neq X_t;\\
 \tau_{X_t}(\rho_t(X_t))  & \text{ if } x =X_t,
\end{cases}\\
X_{t+1}&:=\tau_{X_t}(\rho_t(X_t)).
\end{split}
\end{equation}
The walk is immediately terminated if the walker reaches a vertex in the \emph{sink} $Z$.
Note that it is possible for a walk to never terminate.

A rotor walk is \emph{transient} if every vertex of $G$ is visited by the walker at most finitely many times, and is \emph{recurrent} otherwise.

One aspect of the rotor walk that we will study in this paper is the final rotor configuration of a transient walk, defined as follows.
 
\begin{definition}[Final rotor configuration]\label{definition: final rotor configuration}
The \emph{final rotor configuration} $\sigma(\rho):=\sigma_{G,Z}(a,\rho)$ of a transient rotor walk is  given by 
\begin{equation*}
\sigma(\rho)(x):=\lim_{t \to \infty} \rho_t(x) \qquad \forall x \in V(G). \qedhere
\end{equation*}
\end{definition} 
Note that $\sigma(\rho)$ is well defined as the sequence 
$(\rho_t(x))_{t \geq 0}$ is eventually constant by the assumption that the walk is transient.

 
 Another aspect of the rotor walk that we will study in this paper is the odometer, defined as follows.
 
 \begin{definition}[Odometer]\label{definition: odometer}
The \emph{odometer} $u_{G,Z}(a,\rho, x)$ is the   number of visits to $x$ strictly before hitting $Z$ by the rotor walk with initial location $a$ and initial rotor configuration $\rho$, i.e.
\[u_{G,Z}(a,\rho, x):=|\{t \geq 0 \mid X_t=x \}|. \qedhere  \]
\end{definition} 
Note that the odometer for $x \in Z$ is always equal to $0$ as the odometer only counts visits strictly before hitting $Z$.


We will compare the odometer of the rotor walk to the Green function, which is the odometer for the simple random walk..

\begin{definition}[Green function]\label{definition: Green function}
The \emph{Green function}  
 $\Gc_{G,Z}(a,x)$ is the expected number of visits to $x$  strictly before hitting $Z$ by the simple random walk on $G$ that starts at $a$.
 \end{definition}

We will also study the following extended notion of odometer that we call  occupation rate.
\begin{definition}[Occupation rate]\label{definition: occupation rate}
For any $n\geq 1$, 
we define
\[ S_{G,Z,n}(a,\rho,x):= \sum_{i=0}^{n-1} {u_{G,Z}(a,\sigma^i(\rho), x)},  \]
if the rotor walks with $\rho, \sigma(\rho),\ldots, \sigma^{n-1}(\rho)$ as the initial rotor configuration are all transient,
and $S_{G,Z,n}(a,\rho,x):=\infty$ otherwise.
That is, $S_{G,Z,n}(a,\rho,x)$ is the  total  number of visits to $x$ of $n$ rotor walks performed without resetting the rotors in between walks.
%
The $n$-th \emph{occupation rate} of the rotor walk  is  $\frac{S_{G,Z,n}(a,\rho,x)}{n}$.
\end{definition}


We will omit  the underlying graph $G$, the initial location $a$, the initial rotor configuration $\rho$, or the sink $Z$  from the notations when they are evident from the context.
In particular, we will always omit the initial location $a$ from the notation.
 
\section{Rotor walks on finite graphs}\label{section: RWLMs for finite graphs}
In this section we review several results  for rotor walks on finite graphs, and we refer to \cite{HLM08} for a more detailed discussion on this topic.

 Here $G$ is a finite simple connected graph;
 the initial location of the walker is a fixed vertex $a$;
 and the sink $Z$ is a nonempty subset of $V(G)$.
 Note that the corresponding rotor walk always terminates in finite time, as the walker will eventually reach a vertex in $Z$.

%


 


The  initial rotor configuration for the rotor walk  is picked from oriented  spanning forests, defined as follows.

\begin{definition}[Oriented spanning forest]
A \emph{$Z$-oriented spanning forest} of $G$ is an oriented subgraph $F$ of $G$ such that
\begin{enumerate}
\item Every vertex in $Z$ has outdegree $0$ in $F$;
\item Every vertex in $G \setminus Z$ has outdegree 1 in $F$; and
\item $F$ contains no directed cycles. \qedhere
\end{enumerate}
%
%
\end{definition}

Note that  each $Z$-oriented spanning forest $F$ corresponds to a rotor configuration $\rho:=\rho_F$,
where for every $x \in V(G) \setminus Z$, the state $\rho(x)$ is the out-neighbor of $x$ in $F$. 
 Throughout this paper, we will treat $\rho$ both as  a rotor configuration and as an oriented subgraph of $G$ interchangeably.
 
 We denote by $\SF(G,Z)$ the set of $Z$-oriented spanning forests of $G$.
 
 \begin{definition}[Oriented uniform spanning forest]\label{definition: usf}
 The \emph{$Z$-oriented uniform spanning forest}, denoted by $\ousf(G,Z)$, is the uniform probability distribution on $Z$-oriented spanning forests of $G$.
 \end{definition}
 

The next proposition shows that  $\ousf(G,Z)$ is in a certain sense a stationary distribution of the rotor walk. 
Recall the definition of the final rotor configuration $\sigma(\rho)$ (Definition~\ref{definition: final rotor configuration}).

\begin{proposition}[{\cite[Lemma~3.11]{HLM08}}]\label{proposition: stationarity of wsf for finite graphs}
Let $G$ be a finite simple connected graph.
Consider any rotor walk on $G$ with initial location $a$ and with nonempty sink $Z$.
If  the initial rotor configuration $\rho$ is sampled from $\ousf(G,Z)$,
then the final rotor configuration $\sigma(\rho)$ also follows the law of  $\ousf(G,Z)$. \qed
\end{proposition}
 
%
%
%
%
%
 

The next proposition shows that the expected number of visits by the rotor walk and the simple random walk are equal if
the initial rotor configuration is sampled from $\ousf(G,Z)$.
Recall the definition of the odometer $u$ (Definition~\ref{definition: odometer}) and the Green function $\Gc$ (Definition~\ref{definition: Green function}).

 \begin{proposition}\label{proposition: Green function equality  for finite graphs}
Let $G$ be a finite simple connected graph.
Consider any rotor walk on $G$ with  initial location $a$ and  with nonempty sink $Z$.
Then,  for all $x \in V(G)$,
\begin{equation*}
 \Eb_{\rho}[u(\rho,x)] =\Gc(x), 
\end{equation*}
where $\rho$ is sampled from $\ousf(G,Z)$.
\end{proposition}
Note that  links between the Green function and the dynamics of the process  have appeared regularly in the study of self-organized criticality; see \cite{Dhar90, HLM08, HP10, CL18} for  non-exhaustive examples.

We now build toward the proof of Proposition~\ref{proposition: Green function equality  for finite graphs}.

\begin{lemma}\label{lemma: odometer is harmonic for finite graphs}
Let $G$ be a finite simple connected graph.
Consider any rotor walk on $G$ with initial location $a$ and with nonempty sink $Z$.
Then, for any rotor configuration $\rho$ and any $x \in V(G)$,
\begin{equation*}
  \lim_{n \to \infty}\frac{S_n(x)}{n} =\Gc(x).
  \end{equation*}
\end{lemma}

We will use the following notation in the proof of Lemma~\ref{lemma: odometer is harmonic for finite graphs}.
For any function $f: V(G)\to \Rb$, the 
  \emph{discrete Laplacian} of $f$ is  the function 
\[ \Delta f(x) :=\frac{1}{\deg(x)} \sum_{y \sim x} f(y)-f(x) \qquad \forall \ x \in V(G).  \]
Here $y\sim x$ means that $y$ is a neighbor of $x$ in $G$.
For any $x \in V(G)$ and any $y \sim x$,
we denote by $u(\rho,y,x):=u_{G,Z}(\rho,y,x)$ the  total number of utilization of the edge $(y,x)$ by the rotor walk, i.e.,
\[ u(\rho,y,x):=|\{t \geq 0 \mid   X_t=y \text{ and } X_{t+1}=x \}|.  \] 
For  $n\geq 1$,
we denote by $S_n(\rho,y,x):=S_{G,Z,n}(\rho,y,x)$ the  total number of utilization of the edge $(y,x)$ by $n$ rotor walks performed sequentially, i.e.,
\[S_n(\rho,y,x):=\sum_{i=0}^{n-1}   u(\sigma^i(\rho),y,x). \]

\begin{proof}[Proof of Lemma~\ref{lemma: odometer is harmonic for finite graphs}]
Since $G$ is a finite graph,
the sequence  $(\sigma^i(\rho))_{i\geq 0}$ is eventually periodic, i.e., there exist integers $k$ and $m$ such that 
$\sigma^k(\rho)=\sigma^{k+m}(\rho)$.
We can without loss of generality assume that this sequence is periodic (by replacing $\rho$ with $\sigma^k(\rho)$ if necessary).
This implies that the sequence 
$(u(\sigma^n(\rho),x))_{n \geq 0}$ is also periodic, which in turn implies that
\begin{align}\label{equation: occupation rate limit}
 \lim_{n \to \infty} \frac{S_{n}(\rho,x)}{n}=&\lim_{n \to \infty} \frac{1}{n} \sum_{i=0}^{n-1}{u(\sigma^{i}(\rho),x)}
 =\frac{S_m(\rho,x)}{m}. 
\end{align}


Let $F: V(G) \to \Rb$ be the function given by $F(x):=\frac{S_m(\rho,x)}{m \deg(x)}$.
It suffices to show that $F$
 satisfies the following identities:
%
%
\begin{equation}\label{equation: odometer is harmonic}
\begin{split}
 \Delta F(x)=&  \,  -\mathbbm{1}\{a=x\} / \deg(x) \\
  F(x)=& \, 0  
 \end{split}
 \qquad 
 \begin{split}
   &\text{for } x \notin Z; \qquad \text{and}\\
   &\text{for } x \in Z.
 \end{split}
\end{equation}
Indeed, this is because the function $\Gc(x)$ also satisfies the same identities (see \cite[Proposition~2.1]{LP16} for a proof).
By the uniqueness principle for the Dirichlet problem on finite graphs,
we then conclude that 
 $F(x)= \frac{\Gc(x)}{\deg(x)}$, which together with \eqref{equation: occupation rate limit} implies the lemma.

The identity that $F(x)=0$ for $x\in Z$ is a consequence of   the odometer counting only visits strictly before hitting $Z$.
We now prove the identity $ \Delta F(x)=    -\mathbbm{1}\{a=x\} / \deg(x)$ for $x \notin Z$. 
Note that the total number of visits  to any vertex $x \notin Z$ of the rotor walk is equal to the  total number of utilization  of its  incoming edges if $x$ is not equal to $a$, and is equal to the same number but with one extra visit if $x=a$ (because of the visit to $a$ at the $0$-th step).
This implies that, for any  $x\notin Z$, 
\begin{equation}\label{equation: odometer 1}
S_m(\rho,x)=m \mathbbm{1}\{a=x   \} +  \sum_{y \sim x} S_m(\rho,y,x).
\end{equation}

Now note that
we have the final rotor configuration $\sigma^m(\rho)$ after performing $m$ rotor walks is equal to the initial rotor configuration $\rho$.
Since the local mechanism at $y$ is a periodic function with period $\deg(y)$, it then follows that $S_m(\rho,y,x)= S_m(\rho,y)/\deg(y)$.
Plugging this into \eqref{equation: odometer 1} and dividing both sides by $m\deg(x)$, we then get
\begin{equation*}
\frac{S_m(\rho,x)}{m\deg(x)}= \frac{\mathbbm{1}\{a=x   \}}{\deg(x)} +  \frac{1}{\deg(x)}\sum_{y\sim x} \frac{S_m(\rho,y)}{m\deg(y)}.
\end{equation*}
Note that this equation is equivalent to $ \Delta F(x)=    -\mathbbm{1}\{a=x\} / \deg(x)$.
This completes the proof.
%
\end{proof}

We now present the proof of Proposition~\ref{proposition: Green function equality  for finite graphs}.

\begin{proof}[Proof of Proposition~\ref{proposition: Green function equality  for finite graphs}]
We have for any $n \geq 1$ that
\begin{align*}
\Eb_{\rho}\left[\frac{S_n(\rho,x)}{n} \right]=&\frac{1}{n} \sum_{i=0}^{n-1}
\Eb_{\rho}\left[ {u(\sigma^i(\rho),x)}\right]=\frac{1}{n} \sum_{i=0}^{n-1}
\Eb_{\rho}\left[ {u(\rho,x)}\right] \\
=&\Eb_\rho[u(\rho,x)],
\end{align*}
where the second equality is due to Proposition~\ref{proposition: stationarity of wsf for finite graphs}.
It then follows that
\[ \Eb_\rho[u(\rho,x)]=\lim_{n \to \infty} \frac{S_n(\rho,x)}{n}=\Gc(x),  \]
where the second equality is due to Lemma~\ref{lemma: odometer is harmonic for finite graphs}.
This proves the proposition.
\end{proof}

 \section{Wired spanning forest and rotor walks}
 \label{section: transience of rotor walks}
In this section we begin our investigation of  rotor walks
whose initial rotor configuration is sampled from the oriented wired uniform spanning forest,
and in the process we prove Theorem~\ref{theorem: odometer is bounded above by Green function}. 

For the rest of this paper,  $G$ is a simple connected graph that is locally finite and transient, 
the initial location of the walker is a fixed vertex $a$,
and 
 the sink $Z$ for the rotor walk is empty (i.e. the walk is never terminated), unless stated otherwise.
The initial rotor configuration is picked  from oriented spanning forests of $G$, defined as follows.

\begin{definition}[Oriented spanning forests]\label{definition: oriented spanning forest}
An \emph{oriented spanning forest} of $G$ is an oriented subgraph $F$ of $G$ such that
\begin{itemize}
\item Every vertex of $G$ has outdegree exactly 1 in $F$; and
\item There are no directed cycles in $F$. \qedhere
\end{itemize}
\end{definition} 
 We denote by $\SF(G)$ the set of oriented spanning forests of $G$.
 

An \emph{exhaustion} of $G$ is a finite sequence $(W_r)_{r\geq 0}$ of increasing finite connected subsets of $V(G)$ such that
$\bigcup_{r \geq 0} W_r=V(G)$.
Let $G_r$ be the induced subgraph of $W_r$, and let   $Z_r$ be the set
\[Z_r:=\{x \in W_r \mid  d_{G}(x, G\setminus W_r)=1\}.   \]
That is, $Z_r$ is the set of vertices in $W_r$ that are adjacent to a vertex not in $W_r$.
We denote by $\mu_r$ the probability measure $\ousf(G_r,Z_r)$ (see Definition~\ref{definition: usf}) on the oriented spanning trees of $G_r$.

\begin{definition}[Oriented wired uniform spanning forest]
\label{definition: wusf}
The \emph{wired uniform spanning forest oriented toward infinity} $\owusf(G):=\owusf$ is the probability distribution on oriented subgraphs of $G$ such that, for any finite subset $B$ of directed edges of $G$, 
\begin{equation}\label{equation: limit definition wusf} 
\owusf[B \subseteq F]=\lim_{r \to \infty} {\mu}_r[B \subseteq {F_r}],
\end{equation}
where $F$ is an oriented subgraph of $G$ sampled from $\owusf(G)$, and ${F_r}$ is an $Z_r$-oriented subgraph of $G_r$ sampled from ${\mu}_r$.
\end{definition}

The limit in \eqref{equation: limit definition wusf} exists and does not depend on the choice of the exhaustions (see \cite[Theorem~5.1]{BLPS01} or  \cite[Proposition~10.1]{LP16} for a proof).
Note that the assumption that $G$ is transient is crucial here,
as  $\lim_{r \to \infty} {\mu}_r[B \subseteq {F_r}]$ can depend on the choice of exhaustions if the underlying graph is recurrent (Importantly, the choice of exhaustions influences the orientation of $F$, but not the underlying graph of $F$!).

Throughout this paper we will fix our choice of $W_r$  by taking 
$W_r$ to be the ball $B_r$ of radius $r$ centered at $a$ (i.e., the set of vertices whose graph distance from $a$ is at most $r$).
Note that
$Z_r$ is then equal to the boundary $\partial B_r$ of the ball $B_r$  (i.e., the set of vertices whose graph distance from $a$ is equal to $r$).

We remark that $\owusf(G)$ can also be constructed by using Wilson's method oriented toward infinity. Importantly, 
we do not remove the orientation of  the edges in the construction.
 We refer to \cite{BLPS01,LP16} for a more detailed discussion on the wired uniform spanning forest.

Note that 
 every vertex of $G$ has outdegree $1$ in the oriented subgraph $F$ sampled from $\owusf(G)$.  
In particular, $F$ corresponds to the rotor configuration $\rho:=\rho_F$ where for every $x \in V(G)$ the state $\rho(x)$ is the out-neighbor of $x$ in $F$.
As has been mentioned in the beginning of the section,
our initial rotor configuration will always be sampled from $\owusf(G)$, unless stated otherwise.


We now restate Theorem~\ref{theorem: odometer is bounded above by Green function} for the convenience of the reader.
Recall the definition of the odometer $u$ (Definition~\ref{definition: odometer}) and the Green function $\Gc$ (Definition~\ref{definition: Green function}).
Note that $\Gc(x)$ is always finite  since $G$ is a transient graph.

\begin{reptheorem}{theorem: odometer is bounded above by Green function}
Let $G$ be a simple connected graph that is locally finite and transient.
Consider any rotor walk on $G$ with initial location $a$ and with empty sink.
Then, for any $x \in V(G)$,
\begin{equation*}
\Eb_{\rho}[u(\rho,x)]\leq \Gc(x), 
\end{equation*}
where $\rho$ is sampled from  $\owusf(G)$.
\end{reptheorem}

The following result is a direct corollary of Theorem~\ref{theorem: odometer is bounded above by Green function}.
 
\begin{corollary}\label{corollary: RWLM is transient}
Let $G$ be a simple connected graph that is locally finite and transient.
Consider any rotor walk on $G$ with initial location $a$ and with empty sink.
Then, for almost every initial rotor configuration sampled from $\owusf(G)$,
the corresponding rotor walk is transient. \qed
\end{corollary}

We now present the proof of Theorem~\ref{theorem: odometer is bounded above by Green function}.

\begin{proof}[Proof of Theorem~\ref{theorem: odometer is bounded above by Green function}]

Let $r$ be any positive integer.
Note that the rotor walk  terminated upon hitting  $Z_r=\partial B_r$ is a process that depends only on the rotor of vertices  in 
$W_r=B_r$.
In particular, the  number  of visits to $x$  by this rotor walk is a function of $\rho$ that depends only on finitely many edges.
By  \eqref{equation: limit definition wusf}, we then have
\begin{align}\label{equation: rotor walk is transient 1}
\Eb_\rho[u_{G,Z_r}(\rho,x)]= \lim_{R \to \infty} \Eb_{\rho_R}[u_{G_R,Z_r}(\rho_R,x)],
\end{align}
where $\rho_R$ is a rotor configuration of $G_R$ sampled from 
$\ousf(G_R,Z_R)$.

Now note that the number of visits to any vertex will only increase if the sink of the rotor walk  is moved further away from the initial location of the walker.
Hence, for any $R\geq r$, we have
\begin{align}\label{equation: rotor walk is transient 2}
\begin{split}
 \Eb_{\rho_R}[u_{G_R,Z_r}(\rho,x)] \leq 
  \Eb_{\rho_R}[u_{G_R,Z_R}(\rho,x)]
  =& \Gc_{G_R,Z_R}(x), \end{split}
\end{align}
where the equality is due to the stationarity of $\ousf(G_R,Z_R)$ for rotor walks on finite graphs~(Proposition~\ref{proposition: Green function equality  for finite graphs}).
Combining  \eqref{equation: rotor walk is transient 1} and \eqref{equation: rotor walk is transient 2} and then taking the limit as $R\to \infty$, we then have
\begin{align*}
\Eb_\rho[u_{G, Z_r}(\rho,x)] \leq  \lim_{R \to \infty} \Gc_{G_R,Z_R}(x)=  \Gc_{G,\varnothing}(x).
\end{align*}

Now note that $u_{G, Z_r}$ increases to $u_{G,\varnothing}$
as $r \to \infty$ (because the total number of visits can only increase if the sink is further away).
By the monotone convergence theorem, we then conclude that:
\begin{align*}
\Eb_{\rho}[u_{G,\varnothing}(\rho,x)]= &  \lim_{r \to \infty} \Eb_\rho [u_{G, Z_r}(\rho,x)]
\leq  \Gc_{G,\varnothing}(x),
\end{align*}
as desired.
\end{proof}

Using a similar method in proving  Theorem~\ref{theorem: odometer is bounded above by Green function},
one can prove the following stronger result.
Recall the definition of occupation rate $S_n/n$ from Definition~\ref{definition: occupation rate}.

\begin{proposition}\label{proposition: n-th odometer is bounded above by Green function}
Let $G$ be a simple connected graph that is locally finite and transient.
Consider $n$ rotor walks on $G$ performed sequentially with initial location $a$ and with empty sink.
Then, for any  $x \in V(G)$, 
\[ \Eb_{\rho}[S_n(\rho,x)]\leq n \Gc(x), \]
where $\rho$ is sampled from  $\owusf(G)$. \qed
\end{proposition}

\section{Convergence in norm of occupation rates}\label{section: convergence in norm of occupation rate}

In this section we prove Theorem~\ref{theorem: Schramm bound convergence in norm},
which shows that
 the occupation rates of the rotor walk whose initial rotor configuration is sampled from $\owusf(G)$ converges in norm to the Green function.  

We restate Theorem~\ref{theorem: Schramm bound convergence in norm} for the convenience of the reader.
\begin{reptheorem}{theorem: Schramm bound convergence in norm}
Let $G$ be a simple connected graph that is locally finite, transient, and vertex-transitive.
Consider any rotor walk on $G$ with initial location $a$ and with empty sink.
Then, for any $x \in V(G)$, 
\[  \lim_{n \to \infty} \Eb_{\rho} \left[ \left|\frac{S_n(\rho,x)}{n}-\Gc(x)\right| \right] = 0,\]
where $\rho$ is sampled from $\owusf(G)$.
\end{reptheorem}

We now build toward the proof of Theorem~\ref{theorem: Schramm bound convergence in norm}.
The main ingredients  are the  the upper bound for $S_n/n$ from Proposition~\ref{proposition: n-th odometer is bounded above by Green function},
and the 
lower bound for $S_n/n$ from the following lemma.

\begin{lemma}\label{lemma: Schramm bound}
Let $G$ be a simple connected graph that is locally finite.
Consider any rotor walk on $G$ with initial location $a$ and with empty sink.
Then, for any initial rotor configuration $\rho$,
\begin{equation*} 
 \liminf_{n \to \infty} \frac{S_{n}(\rho,x)}{n}\geq   \Gc(x) \qquad \forall x \in V(G).  
\end{equation*}
\end{lemma}

\begin{proof}
Note that if $G$ is a finite graph, then $S_n(\rho,x)=\Gc(x)=\infty$,
and the lemma immediately follows.
We will therefore without loss of generality assume that $G$ is an infinite graph.

Let $r\geq 1$.
Recall that $B_r$ is the set of vertices of $G$ whose graph distance from $a$ is at most $r$, $Z_r$ is the set of vertices whose graph distance  from $a$ is equal to $r$, and $G_r$ is the subgraph of $G$ induced by $B_r$.
Let $\xi$ be the rotor configuration of $G_r$ given by $\xi(x):=\rho(x)$ for all $x \in B_r$.
Now note that the rotor walk on $G_r$ with initial rotor configuration $\xi$  can be coupled  with the rotor walk on $G$ with initial rotor configuration $\rho$,
 provided that both walks are terminated upon hitting $Z_r$.
 Also note that the same observation can be made for the simple random walk on $G_r$ and $G$.
 These observations imply that, for any $x \in B_r$,
 \begin{equation}\label{equation: Schramm bound coupling}
 \frac{S_{G,Z_r,n}(\rho,x)}{n}= \frac{S_{G_r,Z_r,n}(\xi,x)}{n}; \quad \text{and} \quad \Gc_{G,Z_r}(x)=\Gc_{G_r,Z_r}(x).  
 \end{equation}

Now note that $G_r$ is a finite graph  and $Z_r$ is a nonempty set (as $G$ is infinite).
It then follows from Lemma~\ref{lemma: odometer is harmonic for finite graphs} that
 \[ \lim_{n \to \infty} \frac{S_{G_r,Z_r,n}(\xi,x)}{n}= \Gc_{G_r,Z_r}(x).\]
Together with  \eqref{equation: Schramm bound coupling}, this implies that 
\begin{equation}\label{equation: Schramm bound stopped at Z_r}
\lim_{n \to \infty} \frac{S_{G,Z_r,n}(\rho,x)}{n}= \lim_{n \to \infty}\frac{S_{G_r,Z_r,n}(\xi,x)}{n} =\Gc_{G_r,Z_r}(x)=\Gc_{G,Z_r}(x).
\end{equation}

Now  note that occupation rates can only decrease as the sink grows, which gives us ${S_{G,\varnothing,n}(\rho,x)}\geq {S_{G,Z_r,n}(\rho,x)}$. 
Together with  \eqref{equation: Schramm bound stopped at Z_r}, this implies that
\[ \liminf_{n \to \infty}  \frac{S_{G,\varnothing,n}(\rho,x)}{n}\geq \liminf_{n \to \infty} \frac{S_{G,Z_r,n}(\rho,x)}{n}= \Gc_{G_r,Z_r}(x).\]
The lemma now follows by taking the limit of the inequality above as $r \to \infty$.
\end{proof}

We now present the proof of Theorem~\ref{theorem: Schramm bound convergence in norm}.

\begin{proof}[Proof of Theorem~\ref{theorem: Schramm bound convergence in norm}]
Let $\epsilon>0$ be an arbitrary positive real number.
Let $g_{\epsilon}:=\Gc(x)-\epsilon$, and let $A_{n,\epsilon}$ be the set of rotor configurations given by
\[A_{n,\epsilon}:=\left\{ \rho \ \bigg| \  \frac{S_n(\rho,x)}{n}\geq g_\epsilon   \right\}.  \]

Note that
\begin{align*}
&\Eb_{\rho} \left[ \left|\frac{S_n(\rho,x)}{n} -\Gc(x)\right|\right]\leq  \Eb_{\rho} \left[ \left|\frac{S_n(\rho,x)}{n} -g_\epsilon\right|\right] +\epsilon\\
= &\Eb_{\rho} \left[ \mathbbm{1}_{A_{n,\epsilon}}\left(\frac{S_n(\rho,x)}{n} -g_\epsilon\right)\right]
+\Eb_{\rho} \left[ \mathbbm{1}_{A_{n,\epsilon}^c}\left(g_\epsilon- \frac{S_n(\rho,x)}{n} \right)\right]+\epsilon\\ 
=& \Eb_{\rho} \left[ \left(\mathbbm{1}_{A_{n,\epsilon}} -\mathbbm{1}_{A_{n,\epsilon}^c}\right)\frac{S_n(\rho,x)}{n} \right]- g_\epsilon \left(2\Pb_{\rho}[A_{n,\epsilon}]-1 \right)+\epsilon\\
\leq &\Eb_{\rho} \left[ \frac{S_n(\rho,x)}{n} \right]- g_\epsilon \left(2\Pb_{\rho}[A_{n,\epsilon}] -1\right)+\epsilon.
\end{align*}
Together with Proposition~\ref{proposition: n-th odometer is bounded above by Green function},
the inequality above implies that
\begin{equation}\label{equation: convergence in norm}
\lim_{n \to \infty} \Eb_{\rho} \left[ \left|\frac{S_n(\rho,x)}{n} -\Gc(x)\right|\right] \leq \Gc(x)-  g_\epsilon \left(2 \lim_{n \to \infty} \Pb_{\rho}[A_{n,\epsilon}]-1 \right)+\epsilon.
\end{equation}
Now note that 
we have  $\lim_{n \to \infty} \Pb_{\rho}[A_{n,\epsilon}] \to 1$ as $\epsilon \to 0$   by Lemma~\ref{lemma: Schramm bound}.
This implies that the right side of  \eqref{equation: convergence in norm}  tends to $0$ as $\epsilon\to 0$, and the theorem now follows.
\end{proof}
\section{Rotor walk stationarity}
\label{section: rotor walk stationarity}
In this section we continue our investigation of random walks whose initial rotor configuration is sampled from $\owusf(G)$,
and we are interested in checking if $\owusf(G)$ is  a stationary distribution of the rotor walk.



Recall the definition of the final rotor configuration $\sigma(\rho)$ from Definition~\ref{definition: final rotor configuration}.
\begin{definition}[Rotor walk stationarity]\label{definition: RWLM-stationary}
A probability distribution $\mu$ on rotor configurations of $G$ is \emph{rotor walk stationary} with respect to  a given rotor walk if
\begin{enumerate}
\item For almost every rotor configuration $\rho$ sampled from $\mu$, the corresponding rotor walk is transient; and
\item If the initial configuration $\rho$ is sampled from $\mu$, then the final rotor configuration $\sigma(\rho)$ also follows the law of $\mu$.  \qedhere
\end{enumerate}
\end{definition}

The oriented wired uniform spanning forest $\owusf(G)$ satisfies the first condition  by Corollary~\ref{corollary: RWLM is transient},
so it is a natural candidate for a distribution that is  rotor walk stationary.
As it turns out, there are examples for which $\owusf(G)$ is indeed rotor walk stationary (e.g. for rotor walks on the $b$-ary tree $\Tb_b$ $(b \geq 2)$, as we will prove in Section~\ref{section: rotor walk for b-ary trees}),
but there are also examples for which this fails, as shown in Figure~\ref{figure: odometer bound strict inequality}~(Section~\ref{section: intro}).

We now present an extension of Theorem~\ref{theorem: two TFAE} that gives  two different conditions that are equivalent to $\owusf(G)$ being  stationary.
Recall the definition of the odometer $u$ (Definition~\ref{definition: odometer}) and the Green function $\Gc$ (Definition~\ref{definition: Green function}).
\begin{theorem}\label{theorem: three TFAE}
Let $G$ be a   simple connected graph
that is locally finite and transient.
Consider any  rotor walk on $G$ with  initial location $a$ and with empty sink.
The following are equivalent:
\begin{enumerate}
\item[\mylabel{item: S1}{\textnormal{(S1)}}] $\owusf(G)$ is  rotor walk stationary.


\item[\mylabel{item: S2}{\textnormal{(S2)}}] We have $\Eb_\rho[u(\rho,x)]= \Gc(x)$ for any $x \in V(G)$, where $\rho$ is sampled from $\owusf(G)$.

\item[\mylabel{item: S3}{\textnormal{(S3)}}] 
For any $\epsilon >0$ and any $s>0$,
we have for sufficiently large $r$ that
\[ \lim_{R \to \infty} \Pb [\{ X_t^{(R)} \mid t \leq   t_R(s) \} \subseteq   B_r]\geq  1-\epsilon,   \]
where $(X_t^{(R)},\rho_t^{(R)})$ is the rotor walk on $G_R$ with initial location $a$, with initial rotor configuration sampled from $\ousf(G_R,Z_R)$, and with sink $Z_R$.
The integer 
 $t_R(s)$ is the last time this rotor walk visits the ball $B_s$. 

%
\end{enumerate}
\end{theorem}
See Figure~\ref{figure: excursions with three balls} for an illustration of condition \ref{item: S3}.

\begin{figure}[ht!]
\includegraphics[scale=1]{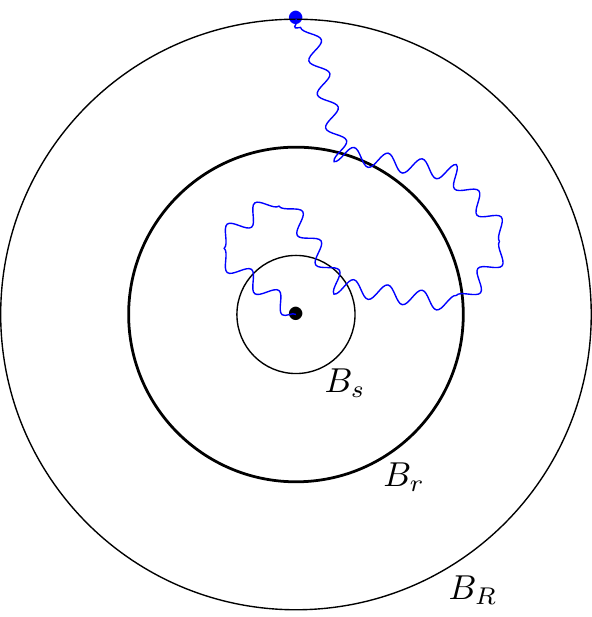}
\caption{An instance of a rotor walk $(X_t^{(R)},\rho_t^{(R)})_{t \geq 0}$ terminated upon visiting the boundary $Z_R$ of the ball $B_R$, where the trajectory of the walker is given by the (blue) squiggly path.
Here  the last visit to the ball $B_s$ is before the first visit to the boundary  $Z_r$ of the ball $B_r$,
and therefore  terminating this walk prematurely upon visiting $Z_r$ (instead of $Z_R$) will not change the rotor of vertices in $B_s$ in the final rotor configuration.
 }
 \label{figure: excursions with three balls}
\end{figure}

Condition~\ref{item: S2} is useful for  deriving other results provided that we already  know that $\owusf(G)$ is rotor walk stationary; 
 Theorem~\ref{theorem: Schramm's bound for rotor walk stationary graphs} will be  proved in  this way.
Condition~\ref{item: S3} is useful for checking  rotor walk stationarity
as it reduces the problem to rotor walks on finite graphs, which is more well-studied in the literature; 
Proposition~\ref{proposition: tree stationarity} in Section~\ref{section: rotor walk for b-ary trees} will be proved in this way.

We now provide a sketch of  how  \ref{item: S2} and \ref{item: S3} imply the rotor walk stationarity of $\owusf(G)$.
The idea is to relate the rotor walk on $G$ to the rotor walk on its exhaustion $(G_R)_{R \geq 0}$.
We first   approximate the rotor walks on those graphs uniformly by the rotor walks that is terminated upon visiting the boundary  of the ball $B_r$  for a fixed radius $r>0$ that is sufficiently large.
The latter walk in turn depends only on rotors of (finitely many) vertices in $B_r$.
It then follows from  \eqref{equation: limit definition wusf} that  the  rotor walk on $G$ with sink $Z_r=\partial B_r$
can be taken as the limit of the rotor walk on $G_R$ with the same sink $Z_r$ as $R \to \infty$.
The stationarity of the wired uniform spanning forest for the rotor walk on $G$ then follows as the consequence of the stationarity of the uniform spannning forest for rotor walks on the finite graphs $(G_R)_{R\geq 0}$ (Proposition~\ref{proposition: tree stationarity}).

The crucial step here  is to find the radius $r>0$ such that the rotor walks on $G_R$ with sink $Z_R$ can be uniformly approximated by the (shorter) rotor walks with sink $Z_r$.
Indeed, we will see  that condition~\ref{item: S2} and \ref{item: S3} are essentially equivalent to requiring that such a radius exists.
Note that such a radius does not always exist, as can be seen from the following example.

\begin{example}
Let $G$ be the $2$-ary tree $\Tb_2$ with an infinite path attached to its root from Figure~\ref{figure: odometer bound strict inequality}.
That is, 
\begin{align*}
V(G):=& V(\Tb_2)  \cup \{y_i \mid i\geq 0  \};\\
E(G):=& E(\Tb_2) \cup \{\{o, y_0\}\}  \cup\{ \{y_i,y_{i+1}\} \mid i \geq 0\},
\end{align*}
where $o$ is the root of $\Tb_2$.

We will perform two rotor walks on $G_R$ $(R\geq 0)$.
Both walks have the same 
  initial location $y_0$ and  the same  initial rotor configuration $\rho_R$ sampled from $\ousf(G_R,Z_R)$, but  with two different choices for the sink; see Figure~\ref{figure: rotor walks with two sinks}.

\begin{figure}[ht!]
\centering
\begin{tabular}{c  }
\includegraphics[scale=1]{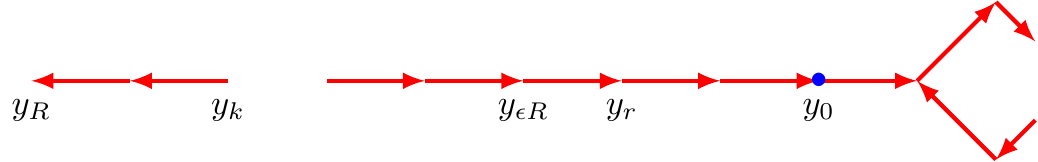}  \\
(a) 
\\ \vspace{0.1 cm}\\
\includegraphics[scale=1]{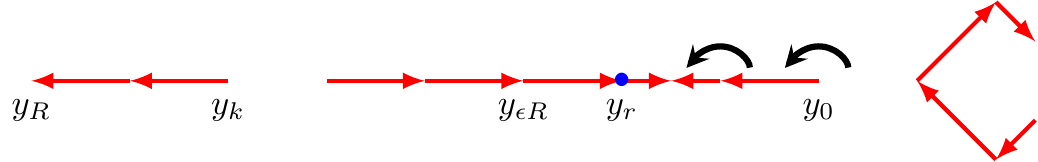}  \\
(b)\\
\\ \vspace{0.1 cm}\\
\includegraphics[scale=1]{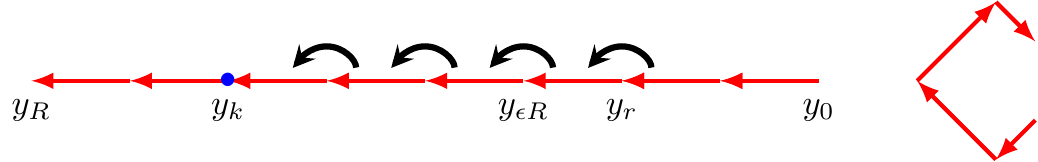}  \\
(c)
\\ \vspace{0.1 cm}\\
\includegraphics[scale=1]{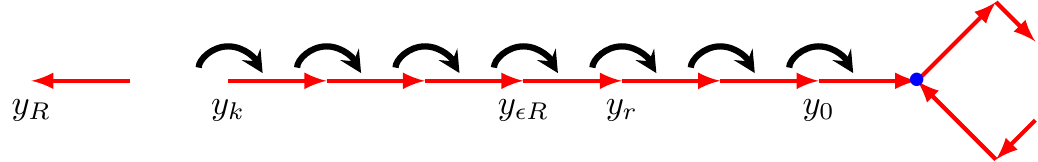}  \\
\end{tabular}
\caption{(a) An initial rotor configuration sampled from $\ousf(G_R,Z_R)$ with a walker initially located at $y_0$.
(b)First, the walker walks toward $y_R$ until it is  stopped at  $y_r \in Z_r$.
(c) Then, the  walker  resumes walking toward $y_R$ until it reaches $y_k$.
(d) Finally, the walker  walks toward the root until it reaches the root.
 }
 \label{figure: rotor walks with two sinks}
\end{figure}

First, consider the rotor walk on $G_R$ terminated upon visiting $Z_r$, where $r$ is a fixed integer.
As $\rho_R$ is sampled from $\ousf(G_R,Z_R)$,
 we have with probability approximately $1-\epsilon$  $(\epsilon>0$) that 
 \[\rho(y_{i+1})=y_{i}  \qquad \forall \ i \leq \epsilon R.\]
It then follows that the walker will walk toward $y_{{R}}$ 
for the first $\epsilon R$ steps of the rotor walk; see Figure~\ref{figure: rotor walks with two sinks}(b).
Since $r \leq \epsilon R$ for sufficiently large $R$,
this rotor walk will terminate in less than $\epsilon R$ steps as it has visited $y_r \in Z_r$ by then.
In particular, 
 this implies that, with probability close to $1$, 
 we have
 \begin{equation}\label{equation: example dynamics that go wrong 1}
  u_{G_R,Z_r}(\rho_R,y_0)=1,   
\end{equation} 
as this walk visits $y_0$ exactly once (namely at the $0$-th step of the walk).

Now, consider the rotor walk on $G_R$ terminated upon visiting $Z_R$.
As $\rho_R$ is sampled from $\ousf(G_R,Z_R)$,
we have 
 with probability approximately $1-\frac{1}{R}$ that
  \[\exists \ i \leq  R \quad \text{s.t.} \quad \rho(y_{i})=y_{i+1}.\]
Let $k$ be the smallest positive integer satisfying this property.
It then follows that
the walker will walk toward $y_R$ for the first $k$ steps of the walk, then turn to  walk toward the root for the next $k+1$ steps; see Figure~\ref{figure: rotor walks with two sinks}(c) and \ref{figure: rotor walks with two sinks}(d).
Also note that this rotor walk will not terminate before the first $2k+1$ steps as it has not visited $Z_R$ yet.
In particular, this implies that, with probability close to 1,
we have 
 \begin{equation}\label{equation: example dynamics that go wrong 2}
  u_{G_R,Z_R}(\rho_R,y_0)\geq 2,   
\end{equation} 
as this walk has visited $y_0$ at least twice 
(namely at the $0$-th and $2k$-th step of the walk). 

Hence we conclude from  \eqref{equation: example dynamics that go wrong 1} and \eqref{equation: example dynamics that go wrong 2} that 
the rotor walks on $G_R$ with sink $Z_R$ cannot be uniformly approximated by the rotor walks with sink $Z_r$
for any fixed $r\geq 0$.
\end{example}

We now build present the proof of the first part of Theorem~\ref{theorem: three TFAE}.
Recall the definition of occupation rate ${S_n}/{n}$ from Definition~\ref{definition: occupation rate}.


\begin{proof}[Proof of \ref{item: S1} implies \ref{item: S2}]
Since  $\owusf(G)$ is rotor walk stationary, we have:
\[  \Eb_\rho[{u(\rho,x)}]=\sum_{i=0}^{n-1}\Eb_\rho \left[ \frac{u(\sigma^i(\rho),x)}{n}\right]=  \Eb_\rho\left[\frac{S_n(\rho,x)}{n}\right]. 
  \]
We also have by  Lemma~\ref{lemma: Schramm bound}
  that 
\[ \liminf_{n \to \infty} \frac{S_n(\rho,x)}{n} \geq \Gc(x),  \]  
for any rotor configuration $\rho$.
These two observations give us:
 \[  \Eb_\rho[{u(\rho,x)}]= \liminf_{n \to \infty}  \Eb_\rho\left[\frac{S_n(\rho,x)}{n} \right]\geq   \Eb_\rho\left[ \liminf_{n \to \infty}\frac{S_n(\rho,x)}{n} \right] \geq  \Gc(x), \]
where the first inequality is due to Fatou's lemma.
  Finally,  we have from Theorem~\ref{theorem: odometer is bounded above by Green function} that
  \[   \Eb_\rho[{u(\rho,x)}] \leq  \Gc(x).\]
  Hence we conclude that  $\Eb_\rho[u(\rho,x)]=\Gc(x)$, as desired.
\end{proof}


We now present the proof of the second part of Theorem~\ref{theorem: three TFAE}.

\begin{proof}[Proof of \ref{item: S2} implies \ref{item: S3}]
%
%
%
%

Let  $u(R,r):=u_{G_R,Z_r}(\rho_R, B_s)$ be the number of visits to the ball $B_s$ by the rotor walk $(X_t^{(R)})_{t \geq 0}$  that is terminated strictly before hitting $Z_r$. 
Note that 
the set 
of vertices visited before the last visit to  $B_s$ 
is contained in the ball $B_r$
if and only if
the walker never comes back to visit $B_s$ after hitting the boundary $Z_r=\partial B_r$ of the ball $B_r$ (see Figure~\ref{figure: excursions with three balls}).
This happens if and only if 
 the number of visits to $B_s$
by the rotor walk terminated upon hitting $Z_r$ is equal to the same number if the rotor walk is not terminated prematurely.
%
%
%
%
%
That is to say, for $r\leq R$,
\[  \Pb [\{ X_t^{(R)} \mid t \leq   t_R(s) \} \subseteq   B_r]=\Eb_{\rho_R}\left[\mathbbm{1}  \left\{u(R,R)- u(R,r) = 0 \right\}  \right].
  \]
Now note that 
\begin{equation*}
\begin{split}
\Eb_{\rho_R}\left[\mathbbm{1}  \left\{u(R,R)- u(R,r) = 0 \right\} \right]  \geq   1- \Eb_{\rho_R}[u(R,R)] - \Eb_{\rho_R}[u(R,r)].
\end{split}
\end{equation*}
It then suffices to show that $\lim_{R \to \infty}\Eb_{\rho_R}[u(R,R)] - \Eb_{\rho_R}[u(R,r)]\leq \epsilon$.

Now note that, we have by the stationarity of $\ousf(G_R,Z_R)$ for rotor walks on finite graphs (Proposition~\ref{proposition: Green function equality  for finite graphs}) that:
  \[\Eb_{\rho_R}[u(R,R)]= \Eb_{\rho_R}[u_{G_R,Z_R}(\rho_R, B_s)]=\Gc_{G_R,Z_R}(B_s).\] 
By taking the limit as $R \to \infty$, we get
\begin{equation}\label{equation: LV FV 2}
 \lim_{R\to \infty}  \Eb_{\rho_R}[u(R,R)]= \Gc_{G,\varnothing}(B_s).  
 \end{equation}

On the other hand, 
 the number of visits to $B_s$ strictly before the walker hits $Z_r$ is an event that only depends on the rotors in the ball $B_r$ (of which there are only finitely many of them).
Hence we have by  \eqref{equation: limit definition wusf} that
\begin{equation*}
 \lim_{R\to \infty} \Eb_{\rho_R}[u(R,r)]=\lim_{R\to \infty}  \Eb_{\rho_R}[u_{G_R,Z_r}(\rho_R, B_s)]=  \Eb_{\rho}[u_{G,Z_r}(\rho, B_s)].
\end{equation*}
Now note that $u_{G,Z_r}$ increases to $u_{G,\varnothing}$ as $r \to \infty$.
By the monotone convergence theorem, we then have for sufficiently large $r$ that 
\[\Eb_{\rho}[u_{G,Z_r}(\rho, B_s)] \geq  \Eb_{\rho}[u_{G,\varnothing}(\rho, B_s)]-\epsilon. \]
Together with condition \ref{item: S2} that  $\Eb_{\rho}[u_{G,\varnothing}(\rho, B_s)]=\Gc_{G,\varnothing}(B_s)$,
the two observations above imply that  
\begin{equation}\label{equation: LV FV 3}
 \lim_{R\to \infty}  \Eb_{\rho_R}[u(R,r)] \geq \Gc_{G,\varnothing}(B_s)-\epsilon. 
\end{equation}

Subtracting \eqref{equation: LV FV 3} from  \eqref{equation: LV FV 2}, we get
\[\lim_{R \to \infty}\Eb_{\rho_R}[u(R,R)] - \Eb_{\rho_R}[u(R,r)]\leq \epsilon, \]
as desired.
\end{proof}

We now present the proof of the last part of Theorem~\ref{theorem: three TFAE}.

\begin{proof}[Proof of \ref{item: S3} implies \ref{item: S1}]
It  suffices to show that
$\Pb_{\rho}[B \subseteq \sigma_G(\rho)]=\Pb_{\rho}[B \subseteq \rho]$
for any  finite set $B$ of directed edges of $G$.

Let $\epsilon>0$ be any positive real number.
Let $s$ be the smallest integer such that all vertices incident to $B$ are contained in $B_s$.
Consider the rotor walk on $G$ with initial rotor configuration $\rho$ and with empty sink.
Since this walk is transient almost surely  (by Corollary~\ref{corollary: RWLM is transient}),
the probability that the walker returns to visit $B_s$ again  after hitting $Z_r=\partial B(a,r)$ converges to $0$ as $r \to \infty$.
Also note that the rotors in the ball $B_s$ will stay constant if the walker never returns to visit $B_s$ again.
Hence, for sufficiently large $r\geq s$, we have:
\begin{equation*}
  \left|\Pb_{\rho}[B \subseteq \sigma_G(\rho)]-\Pb_{\rho}[B \subseteq \sigma_{G,Z_r}(\rho)]\right| \leq \frac{\epsilon}{2}.
  \end{equation*}

Now note that the rotors of $\sigma_{G,Z_r}(\rho)$ in $B_s$ depends only at the rotors of $\rho$ in the ball $B_r$ as the walk is terminated upon hitting $Z_r$.
Since this is a finite set,
we have by  \eqref{equation: limit definition wusf} that
\begin{equation*}
\Pb_{\rho}[B \subseteq \sigma_{G,Z_r}(\rho)]=\lim_{R \to \infty} \Pb_{\rho_R}[B \subseteq \sigma_{G_R,Z_r}(\rho_R)].
\end{equation*}
It then suffices to show that
\begin{equation*}\label{equation: S3 implies S1 2}
\left| \lim_{R \to \infty} \Pb_{\rho_R}[B \subseteq \sigma_{G_R,Z_r}(\rho_R)]- \Pb_{\rho}[B \subseteq \rho] \right| \leq \frac{\epsilon}{2}.
\end{equation*}

Now consider the rotor walk on $G_R$ with initial rotor configuration $\rho_R$ that is terminated upon hitting $Z_R$.
Suppose that  the walk never returns to visit $B_s$ again after it hits $Z_r$.
Then the rotors in the ball $B_s$ of the final rotor configuration remains unchanged even if the walk is terminated  prematurely upon visiting $Z_r$ (see Figure~\ref{figure: excursions with three balls}).
Since $B$ is contained in $B_s$,
this means that $B$ is contained in $\sigma_{G_R,Z_r}(\rho_R)$
 if and only if  $B$ is contained in $\sigma_{G_R,Z_R}(\rho_R)$.
 Hence we have:
 \begin{align*} &\{ X_t^{(R)} \mid t \leq   t_R(s) \} \subseteq   B_r \quad \Rightarrow  \quad \mathbbm{1}\{B \subseteq \sigma_{G_R,Z_r}(\rho_R)\} =\mathbbm{1}\{B \subseteq \sigma_{G_R,Z_R}(\rho_R)\}. 
\end{align*}
Now note that by  \ref{item: S3}  the event $\{ X_t^{(R)} \mid t \leq   t_R(s) \}$  occurs with probability at least $1-\frac{\epsilon}{2}$ for sufficiently large $r$.
It then follows that, for sufficiently large $r$,
 \begin{align*} \left| \Pb_{\rho_R}[ B \subseteq \sigma_{G_R,Z_r}(\rho_R) ]-\Pb_{\rho_R}[B \subseteq \sigma_{G_R,Z_R}(\rho_R)]\right| 
 \leq \frac{\epsilon}{2}.
\end{align*}
On the other hand, the rotor configuration $\sigma_{G_R,Z_R}(\rho_R)$ has the same law as $\rho_R$ by the  
stationarity of $\ousf(G_R,Z_R)$ for rotor walks on finite graphs~(Proposition~\ref{proposition: stationarity of wsf for finite graphs}).
These two facts then imply that:
\begin{equation*}
\left| \Pb_{\rho_R}[ B \subseteq \sigma_{G_R,Z_r}(\rho_R) ]-\Pb_{\rho_R}[B \subseteq \rho_R]\right|\leq \frac{\epsilon}{2}.
\end{equation*}
Taking the limit of the inequality above as $R \to \infty$
and  then applying 
  \eqref{equation: limit definition wusf}  to  $\lim_{R \to \infty} \Pb_{\rho_R}[B \subseteq \rho_R]$, we then conclude that
\begin{equation*}
\left| \lim_{R \to \infty} \Pb_{\rho_R}[B \subseteq \sigma_{G_R,Z_r}(\rho_R)]-  \Pb_{\rho}[B \subseteq \rho] \right| \leq \frac{\epsilon}{2}.
\end{equation*}
This completes the proof.
\end{proof}


\section{A sufficient condition for rotor walk stationarity}\label{section: rotor walk for b-ary trees}
In this section we show that the oriented wired spanning forest is always rotor walk stationary 
for a family of trees that includes 
  the $b$-ary tree $\Tb_{b}$ ($b\geq 2$).
We will need the following notations to describe this family of trees.

Let $\rho$ be a rotor configuration that is an oriented spanning forest of $G$
(recall that we consider $\rho$ both as a rotor configuration and an oriented subgraph of $G$).
An \emph{backward path} (resp. \emph{forward path}) in $\rho$  is a sequence $\langle x_0,x_1,x_2,\ldots\rangle$ such that $\rho(x_{i+1})=x_{i}$ (resp. $\rho(x_{i})=x_{i+1}$)  for every $i \geq 0$.
A path is \emph{infinite} if it contains infinitely many vertices.

Since $\rho$ is an oriented spanning forest,
for each vertex $a$  the subgraph $\rho$ has a unique oriented tree that contains $a$,
and this oriented tree has a unique maximal forward path that starts at $a$.
We denote by $T(a,\rho)$ this unique tree, and by $P(a,\rho)$ this unique maximal forward path.
%
%
%

A vertex $x$ of $G$ is \emph{complete} in $\rho$ if 
$T(x,\rho)$ contains  all neighbors of $x$ in $G$;
and is \emph{incomplete} otherwise. 

%

\begin{proposition}\label{proposition: tree stationarity}
Let $G$ be a tree that is locally finite and transient, and let $a$ be a vertex of $G$.
Consider any rotor walk on $G$ with initial location $a$ and with empty sink.
Suppose that the rotor configuration $\rho$ sampled from $\owusf(G)$ satisfies these two conditions almost surely:
\begin{enumerate}
\item \label{item: tree stationarity 1} $T(a,\rho)$  has no infinite backward path; and
\item \label{item: tree stationarity 2}   There are infinitely many incomplete vertices in $P(a,\rho)$.
\end{enumerate}
Then
$\owusf(G)$ is rotor walk stationary.
\end{proposition}

In order to show that 
 the
 $b$-ary tree $\Tb_{b}$ $(b \geq 2)$ satisfies the two conditions in Proposition~\ref{proposition: tree stationarity}, we need  the following two properties  of the oriented subgraph $\rho$ sampled from $\owusf(\Tb_b)$:
 \begin{enumerate}[label=(\alph*),ref=\alph*]
 \item \label{item: binary tree wsf properties 1} The underlying graph $H$ of any oriented trees of  $\rho$ has exactly one end (i.e. any two infinite unoriented paths in $H$ can differ by at most finitely many vertices) almost surely.  
 \item \label{item: binary tree wsf properties 2} Let $\langle x_0, x_1, x_2, \ldots\rangle $ be the path $P(a,\rho)$, and let $E_i$  $(i\geq 1)$ be the event that $x_i$ is incomplete in $\rho$.
 Then $(E_i)_{i\geq 1}$ are independent events,  and each event has probability  $\Pb(E_i)=1-\left(1/b\right)^{b-1}$ to occur.
 \end{enumerate}
 Indeed, these two properties can be deduced from Wilson's method oriented toward infinity, and
 we refer to~\cite[Section~10.6]{LP16} for proofs.
Now note that conditition~\eqref{item: tree stationarity 1} in Proposition~\ref{proposition: tree stationarity} follows from \eqref{item: binary tree wsf properties 1}, 
and 
conditition~\eqref{item: tree stationarity 2}
follows from \eqref{item: binary tree wsf properties 2}.



We now build toward the proof of Proposition~\ref{proposition: tree stationarity}.
Our proof relies on the following crucial yet simple observation:
If a vertex $x$ was  visited during the walk, then 
$x$ is contained in the same weak component of the final rotor configuration 
as the initial location $a$.


Consider a transient rotor walk $(X_t)_{t\geq 0}$ on $G$.
For any vertex  $x$ of $G$ that was visited by the rotor walk,
 we denote by $\FV(x):=\FV_{G,Z}(\rho,x)$ and $\LV(x):=\LV_{G,Z}(\rho,x)$
the first time and the  last time the vertex  $x$ being visited by the rotor walk, respectively, i.e.
\begin{equation*}\label{equation: first and last visit}
\begin{split}
\FV(x)&:=\min \{ t \geq 0 \mid X_t=x \};\\
\LV(x)&:=\max \{ t\geq 0 \mid X_t=x \}.
\end{split}
\end{equation*}  

\begin{lemma}\label{lemma: first visit is equal to last exit}
Let $G$ be a  tree that is locally finite.
Consider any rotor walk on $G$ with initial location $a$, initial rotor configuration $\rho$, and (not necessarily empty) sink $Z$.
Suppose that this rotor walk is transient, and let $\xi:=\sigma(\rho)$ be the final rotor configuration of this walk.
Then, for any vertex $x_i$ in $P(a,\xi):=\langle x_0,x_1,x_2,\ldots \rangle$ that is incomplete in $\xi$, we have
\[ \FV(x_{i+1})=\LV(x_{i})+1.\]
That is,  the first  visit of  $x_{i+1}$ was right after the last  visit of $x_i$.
\end{lemma}
\begin{proof}
Since $\xi(x_i)=x_{i+1}$, it follows that the walker moved toward $x_i$ right after the last visit to $x_i$ (i.e. $X_t=x_{i+1}$ with $t=\LV(x_i)+1$).
It then suffices to show that $t_1:=\LV(x_i)+1$ is the first visit to  $x_{i+1}$.

Suppose to the contrary that $t_2:=\FV(x_{i+1})$ is strictly smaller than $t_1$.
Now note that $X_{t_2-1}=x_i$ since the unique path from $a$ to $x_{i+1}$ in $G$ goes through $x_i$ (as $G$ is a tree).
Since $X_{t_2-1}=X_{t_1-1}=x_i$ and $t_2<t_1$, it follows from the mechanism of the rotor walk that every neighbor of $x_i$ in $G$ was visited by the walker  in between the $t_2-1$-th and $t_1-1$-th step of the walk.
This implies that every neighbor of $x_i$ is contained in the same component as $x_i$ in the final rotor configuration $\xi$,
and hence $x_i$ is a complete vertex in $\xi$.
This contradicts our assumption that $x_i$ is incomplete in $\xi$,
as desired. 
\end{proof}

For any vertex $x$ of $G$,
we denote by $W(\rho,x)$
the set of vertices of $G$  with a directed path in $\rho$ from the vertex to $x$,
i.e.
\[ W(\rho,x):=\{ y  \mid \exists\,  \langle y=x_0,\ldots, x_n=x\rangle  \text{ s.t. }  \rho(x_i)=x_{i+1} \ \forall \ i<n \}. \]


\begin{lemma}\label{lemma: range of a rotor walk}
Let $G$ be a  tree that is locally finite.
Consider any rotor walk $(X_t,\rho_t)_{t\geq 0}$ on $G$ with initial location $a$, initial rotor configuration $\rho$, and  (not necessarily empty) sink Z.
Suppose that the rotor walk is transient, and let $\xi:=\sigma(\rho)$ be the final rotor configuration of this walk.
Then, for any vertex $x$ in $P(a,\xi)$ that is incomplete in $\xi$, we have
\[ \{X_t \mid t \leq \LV(x)  \} \subseteq W(\xi,x).\]
\end{lemma}
\begin{proof}
Let $P(a,\xi):=\langle x_0,x_1,x_2,\ldots \rangle$, and let $x=x_i$ be an incomplete vertex in $\xi$.
By Lemma~\ref{lemma: first visit is equal to last exit},
the walker had not visited $x_{i+1}$ yet during the first $\LV(x_i)$-th step of the walk.
Since $G$ is a tree, this means that, during  the first $\LV(x_i)$-th step of the walk,  the walker has only visited vertices in the weak component of $G\setminus \{x_i,x_{i+1}\}$ that contains $x_i$.
On the other hand, all vertices visited by the walker are in the weak component of $a$ in $\xi$.
Now note that 
the intersection of these two components is equal to $W(\xi,x_i)$,
and the lemma now follows.
\end{proof}

%

We now present the proof of Proposition~\ref{proposition: tree stationarity}.
\begin{proof}[Proof of Proposition~\ref{proposition: tree stationarity}]
It suffices to check that condition~\ref{item: S3} in Theorem~\ref{theorem: three TFAE} is satisfied.
That is, for any $\epsilon >0$ and any $s>0$,
we have for sufficiently large $r$ that
\[ \lim_{R \to \infty} \Pb [\{ X_t^{(R)} \mid t \leq   t_R(s) \} \subseteq   B_r]\geq  1-\epsilon,   \]
where $(X_t^{(R)},\rho_t^{(R)})$ is the rotor walk  on $G_R$ with initial location $a$, with initial rotor configuration $\rho_R$ sampled from $\ousf(G_R,Z_R)$, and with sink $Z_R$.
The integer $t_R(s)$ is the last visit of $B_s$.

Let $\xi_R:=\sigma_{G_R,Z_R}(\rho_R)$ be the final rotor configuration of the rotor walk $(X_t^{(R)},\rho_t^{(R)})$.
Note that $\xi_R\overset{d}{=} \ousf(G_R,Z_R)$ by the rotor walk stationarity of $\ousf(G_R,Z_R)$ for finite graphs (Proposition~\ref{proposition: stationarity of wsf for finite graphs}).

Fix $r \geq 0$.
For any rotor configuration $\rho$,
let $E_r(\rho)$ be the event that that there exists a vertex $x$ such that
\begin{enumerate}[label=(\alph*)]
\item  $x$ is an incomplete vertex in $\rho$ that is contained in $P(a,\rho) \cap (B_r \setminus B_s)$;  and
\item  $W(\rho,x)$ is contained in $B_r$.
\end{enumerate}
Note that
the event $E_r(\rho)$ depends only on edges in $B_{r+1}$,
and hence we have by  \eqref{equation: limit definition wusf} that 
\begin{equation}\label{equation: d-ary tree 1}
 \lim_{R \to \infty} \Pb_{\xi_R}[E_r(\xi_R)]=\Pb_{\xi}[E_r(\xi)],  
\end{equation}
where $\xi\overset{d}{=} \owusf(G)$.

Since $\xi$ satisfies condition~\eqref{item: tree stationarity 1} in the proposition,
we have that there exists an incomplete vertex $x$ in $\rho$
that is  contained in $P(a,\rho) \cap (B_{r} \setminus B_s)$,
where $r$ is any integer greater than a constant $r_1(\xi)>0$ that depends on $\xi$.
Since $\xi$ satisfies condition~\eqref{item: tree stationarity 2}
 in the proposition,
 we also have that $W(\rho,x)$ is contained in $B_r$,
 where $r$ is any integer greater than a constant $r_2(\xi)>0$ that depends on $\xi$.
 Since $r_1(\xi)$ and $r_2(\xi)$ are almost surely finite, 
 we have for sufficiently large $r$ that 
\begin{equation}\label{equation: d-ary tree 2}
\Pb_{\xi}[E_r(\xi)]\geq \Pb_\xi[r_1(\xi),r_2(\xi)<r] \geq 1-\epsilon. 
\end{equation}
Combining  \eqref{equation: d-ary tree 1} and \eqref{equation: d-ary tree 2}, we then get
\begin{equation}\label{equation: d-ary tree 3}
 \lim_{R \to \infty} \Pb_{\xi_R}[E_r(\xi_R)]\geq 1-\epsilon.  
\end{equation}

Now note that, if $E_r(\xi_R)$ occurs, then we have
by Lemma~\ref{lemma: range of a rotor walk} that
the range of the rotor walk on $G_R$ is contained $W(\xi_R,x)$, which in turn is contained in the ball $B_r$, 
i.e.
\[ \{X_t^{(R)} \mid t \leq \LV(x)  \}\subseteq W(\xi_R,x)\subseteq B_r. \]
It then follows from  \eqref{equation: d-ary tree 3} that
\[  \lim_{R \to \infty} \Pb[\{X_t^{(R)} \mid t \leq \LV(x)  \}\subseteq B_r]\geq   \lim_{R \to \infty}   \Pb_{\xi_R}[E_r(\xi_R)]\geq 1-\epsilon,     \]
and the proof is complete.
\end{proof}

\section{Almost sure convergence of  occupation rates}\label{section: occupation rate}


In this section we show that  occupation rates of rotor walks  converge to the Green function under assumptions of  Theorem~\ref{theorem: Schramm bound for vertex-transitive graphs} or Theorem~\ref{theorem: Schramm's bound for rotor walk stationary graphs}.

We will first present the proof of Theorem~\ref{theorem: Schramm's bound for rotor walk stationary graphs} (as it has a simpler proof).
We restate the theorem here for the convenience of the reader.
Recall the definition of occupation rate ${S_n}/{n}$ (Definition~\ref{definition: occupation rate}) and Green function $\Gc$ (Definition~\ref{definition: Green function}).

\begin{reptheorem}{theorem: Schramm's bound for rotor walk stationary graphs}
Let $G$ be a connected simple graph that is locally finite and transient. 
Consider any rotor walk on $G$ with initial location $a$ and with empty sink.
Suppose that  $\owusf(G)$ is rotor walk stationary.
Then, for  almost every $\rho$ picked from $\owusf(G)$,
\[  \lim_{n\to \infty} \frac{S_n(\rho,x)}{n}=\Gc(x) \qquad \forall \ x \in V(G). \]
\end{reptheorem}

\begin{proof}
First note that  $\sigma$ is a function on rotor configurations that is measure preserving with respective to
 $\owusf(G)$ (by the assumption that $\owusf(G)$ is rotor walk stationary).
 Also note that  $u(\cdot,x)$ is integrable with respect to the measure $\owusf(G)$ (by Theorem~\ref{theorem: odometer is bounded above by Green function}).
 It then follows from Birkhoff-Khinchin theorem (otherwise known as the pointwise ergodic theorem) that
 the limit
 \[  X(\rho):= \lim_{n \to \infty} \frac{S_n(\rho,x)}{n}= 
 \lim_{n \to \infty} \frac{1}{n} \sum_{i=0}^{n-1} u(\sigma^i(\rho),x),\]
 exists for almost every $\rho$ sampled from $\owusf(G)$,
 and furthermore $\Eb_{\rho}[X(\rho)]=\Eb_{\rho}[u(\rho,x)]$.
It then suffices to show that $X(\rho)=\Gc(x)$ almost surely.

Since $\owusf(G)$ is rotor walk stationary, we have by Theorem~\ref{theorem: two TFAE}  that
\begin{equation}\label{equation: Schramm bound stationary 1}
\Eb_{\rho}[X(\rho)]=\Eb_{\rho}[u(\rho,x)]=\Gc(x).
\end{equation}
 On the other hand, we have by Lemma~\ref{lemma: Schramm bound} that, for any $\rho$,
 \begin{equation}\label{equation: Schramm bound stationary 2}
X(\rho)=\lim_{n \to \infty} \frac{S_n(\rho,x)}{n}= \liminf_{n \to \infty} \frac{S_n(\rho,x)}{n}\geq \Gc(x).
\end{equation}
It then follows from  \eqref{equation: Schramm bound stationary 1} and \eqref{equation: Schramm bound stationary 2} that $X(\rho)=\Gc(x)$ almost surely, as desired.
%
%
\end{proof}

We now present the proof of Theorem~\ref{theorem: Schramm bound for vertex-transitive graphs}, and we restate the theorem here for the convenience of the reader.

\begin{reptheorem}{theorem: Schramm bound for vertex-transitive graphs}
Let $G$ be a simple connected graph that is  locally finite, transient, and vertex-transitive.
Consider any rotor walk on $G$ with initial location $a$ and with empty sink.
Then, for almost every $\rho$ sampled from $\owusf(G)$,
\[\lim_{n \to \infty} \frac{S_n(\rho,x)}{n}=\Gc(x) \qquad \forall \ x \in V(G).   \]
\end{reptheorem}

%

We now build toward the proof of Theorem~\ref{theorem: Schramm bound for vertex-transitive graphs}.
We will use the following lower bound for $S_n/n$ that holds for all vertex-transitive graphs.
We would like to warn the reader  that this bound is far from sharp, but is sufficient for our purpose.

\begin{lemma}\label{lemma: occupation rate estimate with neutral green function}
Let $G$ be a simple connected graph that is locally finite, transient, and vertex-transitive.
Consider any rotor walk on $G$ with initial location $a$ and with empty sink.
 Then, for any initial rotor configuration $\rho$ and  any  $n\geq 1$, 
\begin{equation*}
\begin{split}
\frac{S_n(\rho,x)}{n}
\geq & \Gc(x)- C(\log n)^{-2} \qquad \forall \ x \in V(G),
\end{split}
\end{equation*}
where $C>0$ is a constant depending only on $G$.
\end{lemma}

One of the ingredients of the proof of Lemma~\ref{lemma: occupation rate estimate with neutral green function} is the following version of Gromov's theorem~\cite{Gro81} for vertex-transitive graphs by Trofimov~\cite{Tro03}.
Let $V(r):=|B_r|$ be the number of vertices in a ball of radius $r$ in $G$.
Then, for any vertex-transitive graphs,
either $V(r)\asymp  r^D$ for some integer $D$  or  $\lim_{r \to \infty} \frac{V(r)}{r^D}=\infty$  for all integer $D$.
In the former case, we say that $G$ has \emph{polynomial growth} of degree $D$.
In the latter case, we say that $G$ has \emph{superpolynomial growth}.
Here, 
we write $a(r) \lesssim b(r)$ if there exists $c>0$ such that $a(r)\leq c b(r)$ for all $r$, and
we write $a(r) \asymp b(r)$ if $a(r) \lesssim b(r)$ and $b(r) \lesssim a(r)$.

Another ingredient is  
 the  following estimate of the visit probability of the simple random walk, which holds 
 for any vertex-transitive graph with $V(r) \gtrsim r^D$,
\begin{equation}\label{equation: return probability polynomial growth}
p_{t}(a,x)\lesssim t^{-\frac{D}{2}}.
\end{equation}
Here $p_{t}(a,x)$ denotes the probability to visit $x$ at the $t$-th step of the simple random walk on $G$ that starts at $a$.
We refer to  \cite[Corollary~6.32]{LP16} or \cite[Lemma~3.5, Theorem~6.1]{LO17} for a proof.




The final ingredient is  the following estimate of the occupation rate of the rotor walk on vertex-transitive graphs that follows from the proof in \cite[Lemma~8]{FGLP14}:
\begin{equation}\label{equation: occupation rate estimate}
\left| {S_{Z_r,n}(\rho,x)} -n\Gc_{Z_r}(x)\right|\leq \sum_{\substack{x,y \in B_r\\ y \sim x}} |\Gc_{\varnothing}(x)-\Gc_{\varnothing}(y)|.
\end{equation}

%
%

\begin{proof}[Proof of Lemma~\ref{lemma: occupation rate estimate with neutral green function}]

First note that $S_{n}=S_{\varnothing,n}\geq S_{Z_r,n}$ for any $r\geq 0$
as the total number of visits can only decrease if the sink of the rotor walk is enlarged.
This implies that
\begin{equation*}
\begin{split}
& \frac{S_n(\rho,x)}{n}- \Gc_{\varnothing}(x) \geq  \frac{S_{Z_r,n}(\rho,x)}{n}- \Gc_{\varnothing}(x)= K_1+ K_2,
\end{split}
\end{equation*}
where $K_1:= \frac{S_{Z_r,n}(\rho,x)}{n}- \Gc_{Z_r}(x)$ and $K_2:=\Gc_{Z_r}(x)-\Gc_{\varnothing}(x)$.
It then suffices to show that $|K_1| +|K_2| \lesssim (\log n)^{-2}$ for some $r$.

Now note that, for any $r\geq 0$,
\begin{equation*}
\begin{split}
|K_1|\leq & \sum_{\substack{x,y \in B_r\\ y \sim x}}\frac{|\Gc_{\varnothing}(x)-\Gc_{\varnothing}(y)|}{n}  \qquad \text{(by  \eqref{equation: occupation rate estimate})}\\
\leq &  \sum_{\substack{x,y \in B_r\\ y \sim x}}
\frac{\Gc_{\varnothing}(x)+\Gc_{\varnothing}(y)}{n} \\
\lesssim & \, \frac{\Gc_{\varnothing}(B_r)}{n}
\lesssim  \, \frac{V(r)}{n}.
\end{split}
\end{equation*}

Also note that, for any $r \geq 0$,
\begin{align*}
|K_2|=& \Gc_{\varnothing}(x)- \Gc_{Z_r}(x)\leq  \sum_{t\geq r} p_{t}(a,x),
\end{align*}
as the walker has not reached $Z_r=\partial B_r$ yet during the 
first $r$ steps of the simple random walk.

We now consider the case when $G$ has polynomial growth of degree $D$.
Note that $D\geq 3$ since $G$ is transient~(see for example \cite[Theorem~4.6]{SC95} for a proof).
We then have, for any $r\geq 0$,
\begin{align*}
|K_1|+|K_2|\lesssim& \, \frac{V(r)}{n}+ \sum_{t\geq r} p_{t}(a,x)\\
\lesssim & \, \frac{r^D}{n}+ \sum_{t\geq r} t^{-\frac{D}{2}} \qquad \text{(by  \eqref{equation: return probability polynomial growth})}\\
\lesssim& \, \frac{r^D}{n}+ r^{-\frac{1}{2}} \qquad (\text{since } D\geq 3).
\end{align*}
By taking $r=\lfloor n^{\frac{1}{2D}} \rfloor$,  we then 
get $|K_1|+|K_2| \lesssim n^{-\frac{1}{2}}+n^{-\frac{1}{4D}}\lesssim (\log n)^{-2}$, as desired.

We now consider the case when $G$ has superpolynomial growth.
Note that $V(r)\leq e^{cr}$ for some $c >0$ (since $G$ is vertex-transitive)
and $p_{t}(a,a)\lesssim t^{-3}$ by  \eqref{equation: return probability polynomial growth}.
We then have, for any $r\geq 0$,
\begin{align*}
|K_1|+|K_2|\lesssim& \, \frac{V(r)}{n}+ \sum_{t\geq r} p_{t}(a,a)
\lesssim  \, \frac{e^{cr}}{n}+ \sum_{t\geq r} t^{-3} \\
\lesssim& \, \frac{e^{cr}}{n}+ r^{-2}.
\end{align*}
By taking $r=\lfloor \frac{\log n}{2c} \rfloor$,
we then 
get $|K_1|+|K_2| \lesssim n^{-\frac{1}{2}}+ (\log n)^{-2}\lesssim (\log n)^{-2}$, as desired.
\end{proof}
We remark that, in the case of transient Cayley graphs,
one can instead use the inequality  $\Gc_{\varnothing}(B_r) \lesssim r^{5/2}$ from  \cite[Theorem~1.2]{LPS17} to estimate $|K_1|$ and get  a sharper lower bound with polynomial decay in Lemma~\ref{lemma: occupation rate estimate with neutral green function}.
 
We now show that $S_n/n$ converges for any subsequence  that grows exponentially. 
 
\begin{lemma}\label{lemma: escape rate for integer lattice for a specific subsequence}
Let $G$ be a simple transient Cayley graph.
Consider any rotor walk on $G$ with initial location $a$ and with empty sink.
Let $c>1$, and let $n_k:=\lfloor c^k \rfloor$.
Then, for almost every $\rho$ sampled from $\owusf(G)$,
\[\lim_{k \to \infty} \frac{S_{n_k}(\rho,x)}{n_k}=\Gc(x) \qquad \forall \ x \in V(G).   \]
\end{lemma}
\begin{proof}
Write $\varphi(n):=  \Gc_{\varnothing}(a,x)- C(\log n)^{-\frac{1}{2}}$, where $C>0$ is as in Lemma~\ref{lemma: occupation rate estimate with neutral green function}.
Note that $\frac{S_n(a,\rho,x)}{n}-\varphi(n)$ is positive for all  $n$ by Lemma~\ref{lemma: occupation rate estimate with neutral green function}.

Let $\epsilon$ be an arbitrary positive real number.
Then, for $\rho \overset{d}{=} \owusf(G)$, 
\begin{align*}
q_n:=& \Pb_{\rho}\left[\left|\frac{S_n(\rho,x)}{n} -\varphi(n) \right|\geq \epsilon \right] \\
\leq &  \frac{1}{\epsilon} \Eb_{\rho} \left[\frac{S_n(\rho,x)}{n} -\varphi(n)  \right] \qquad \text{(by Markov's inequality)}\\
\leq &\frac{1}{\epsilon} \left(\Gc_{\varnothing}(x)-\varphi(n) \right) \qquad \text{(by Proposition~\ref{proposition: n-th odometer is bounded above by Green function})}\\
 =&  \frac{C}{\epsilon} \, (\log n)^{-2}.
\end{align*}

It then follows that 
\[\sum_{k=1}^\infty q_{n_k}\leq \frac{C}{\epsilon}  (k\log c)^{-2}< \infty.\]
By Borel-Cantelli lemma, we then conclude that,
 \[\limsup_{k \to \infty}  \left|\frac{S_{n_k}(\rho,x)}{n_k} -\varphi(n_k) \right| < \epsilon,  \]
 for  almost every $\rho$ sampled from $\owusf(G)$.
Since the choice of $\epsilon$ is arbitrary and $\varphi(n_k)$ converges to  $\Gc_{\varnothing}(x)$, the lemma now follows.
\end{proof}

We now extend the convergence in Lemma~\ref{lemma: escape rate for integer lattice for a specific subsequence} to the whole sequence.

\begin{proof}[Proof of Theorem~\ref{theorem: Schramm bound for vertex-transitive graphs}]
Let $\epsilon>0$ be an arbitrary positive real number, and let $n_k:=\lfloor (1+\epsilon)^k\rfloor$.
By Lemma~\ref{lemma: escape rate for integer lattice for a specific subsequence},
we have for almost every $\rho$ sampled from $\owusf(G)$ that
\[\lim_{k \to \infty} \frac{S_{n_k}(\rho,x)}{n_k}=\Gc(x). \]

Write $S_n:=S_n(\rho,x)$.
Since $S_n$ is an increasing function of $n$,
we have for any integer $n \in [n_{k}, n_{k+1}]$ that,
\[  \left( \frac{n_k}{n_{k+1}}\right) \frac{S_{n_k}}{n_k}  \ \leq \ \frac{S_n}{n} \ \leq \ \left( \frac{n_{k+1}}{n_{k}}\right) \frac{S_{n_{k+1}}}{n_{k+1}},\]
Since $\frac{n_{k+1}}{n_k} \to 1+\epsilon$ as $k \to \infty$,
we then get
\[  \frac{1}{(1+\epsilon)} \lim_{k \to \infty}\frac{S_{n_k}}{n_k}  \leq  \liminf_{n \to \infty}\frac{S_n}{n} \leq \limsup_{n \to \infty}\frac{S_n}{n} \leq (1+\epsilon)  \lim_{k \to \infty}\frac{S_{n_{k+1}}}{n_{k+1}}. \]
The conclusion of the theorem now follows by applying the inequality above with  $\epsilon$ given by a sequence $\epsilon_1,\epsilon_2,\ldots$ that converges to $0$.
%
%
\end{proof}


\section{Some open questions}\label{section: open problems}
We conclude with a few natural questions:
\begin{enumerate}[label=(\arabic*)]
\item Is  $\owusf(\Zb^d)$ rotor walk stationary with respect to any rotor walk on $\Zb^d$ for $d\geq 3$?

\item Is  the conclusion of Theorem~\ref{theorem: Schramm bound for vertex-transitive graphs} true for all transient graphs? 
That is to say,  does the event 
\[\left\{ \rho \ \bigg | \ \exists \ x \in V(G) \text{ s.t. }  
\limsup_{n \to \infty} \frac{S_n(\rho,x)}{n} > \Gc(x)\right\},\]
always occur with zero probability w.r.t $\owusf(G)$? 

\item Does there exist any rotor configuration $\rho$ for $\Zb^d$ for which its occupation rate converges to a value strictly between $0$ and $\Gc(x)$, i.e.,
\[ \lim_{n \to \infty} \frac{S_n(\rho,x)}{n}=c, \]
where $0<c<\Gc(x)$?
Note that  Landau and Levine~\cite{LL09} showed that such a rotor configuration always exist for any choice of $c$ if  the underlying graph $G$ is the binary tree $\Tb_2$ instead.
\end{enumerate}

\section*{Acknowledgement}
The author would like to thank Lionel Levine and Yuval Peres for their advising throughout the whole project.
In particular, the idea of  using Etemadi's proof of strong law of large numbers for Theorem~\ref{theorem: Schramm bound for vertex-transitive graphs} is due to the suggestion of Peres.
 The author would also like to thank Ander Holroyd  for inspiring discussions, Laurent Saloff-Coste for several references in Section~\ref{section: occupation rate}, and Dan Jerison, Wencin Poh, and Ecaterina Sava-Huss for helpful comments on an earlier draft.
Part of this work was done when the author was visiting the Theory Group at Microsoft Research, Redmond.

\bibliographystyle{alpha}
\bibliography{Schramm}

\end{document}